\documentclass[12pt]{article}
\usepackage{amsmath,enumerate,amsfonts,amssymb,color,graphicx,amsthm}
\usepackage{multirow,caption,subcaption}
\usepackage{hyperref}
\usepackage{todonotes}
\usepackage[normalem]{ulem}
\usepackage{cite}

\numberwithin{equation}{section}

\def\ep{{\epsilon}}

\def\p{{\partial}}

\newcounter{marnote}

\begin{document}
\newtheorem{lem}{Lemma}[section]
\newtheorem{rem}{Remark}
\newtheorem{question}{Question}
\newtheorem{prop}{Proposition}
\newtheorem{cor}{Corollary}
\newtheorem{thm}{Theorem}[section]
\newtheorem{definition}{Definition}[section]
\newtheorem{openproblem}{Open Problem}
\newtheorem{conjecture}{Conjecture}

\newenvironment{dedication}
  {
   \vspace*{\stretch{.2}}
   \itshape             
  }
  {
   \vspace{\stretch{1}} 
  }

\title
{Critical mean field equations for equilibrium turbulence with sign-changing prescribed functions
\thanks{
Email:
sunll@xtu.edu.cn (Linlin Sun);
zhuxiaobao@ruc.edu.cn (Xiaobao Zhu)
}
}
\author{
Linlin Sun\\
School of Mathematics and Computational Science\\
Xiangtan University\\
Xiangtan 411105, P. R. China\\
\\
Xiaobao Zhu\\
School of Mathematics\\
Renmin University of China\\
Beijing 100872, P. R. China\\
}

\date{ }
\maketitle


\begin{abstract}
Let $(M,g)$ be a compact Riemann surface with unit area. We investigate the mean field equation for equilibrium turbulence:
\begin{align}\label{eq-mfe-0}
\begin{cases}
-\Delta u = \rho_1\left(\frac{h_1e^{u}}{\int_Mh_1e^udv_g}-1\right) - \rho_2\left(\frac{h_2e^{-u}}{\int_Mh_2e^{-u}dv_g}-1\right), \\
\int_Mudv_g=0,
\end{cases}
\end{align}
where $\rho_1=8\pi$ and $\rho_2\in(0,8\pi]$ are parameters, and $h_1, h_2$ are smooth functions on $M$ that are positive somewhere. By employing a refined Brezis-Merle type analysis, we establish sufficient conditions of Ding-Jost-Li-Wang type for the existence of solutions to \eqref{eq-mfe-0} in critical cases, particularly when $h_1$ and $h_2$ may change signs. Our results extend Zhou's existence theorems (Nonlinear Anal. 69 (2008), no.~8, 2541--2552) for the case $h_1=h_2\equiv 1$.
 \end{abstract}

\setcounter {section} {0}

\section{Introduction and main theorems}

Let $(M,g)$ be a compact Riemann surface with unit area, $\rho_1, \rho_2 \geq 0$ be parameters, and $h_1, h_2$ be smooth functions on $M$ that are positive somewhere. We study the mean field equation
\begin{align}\label{eq-mfe-1}
\begin{cases}
-\Delta u = \rho_1\left(\frac{h_1e^{u}}{\int_M h_1 e^u dv_g} - 1\right) - \rho_2\left(\frac{h_2e^{-u}}{\int_M h_2 e^{-u} dv_g} - 1\right), \\
\int_M u dv_g = 0,
\end{cases}
\end{align}
where $\Delta$ denotes the Laplace-Beltrami operator and $dv_g$ is the volume form on $(M,g)$.

This equation arises in mathematical physics through the works of Joyce and Montgomery \cite{JM} and Pointin and Lundgren \cite{PL}, where it models the mean field theory for equilibrium turbulence with vortices of arbitrary sign. Further physical background and related results can be found in \cite{Ch,Lions,New} and references therein.

In this paper, we focus on the analytic aspects of equation \eqref{eq-mfe-1}, particularly its existence theory. Before presenting known results, we introduce necessary notations. Let $L^p(M)$ ($p \geq 1$) and $W^{1,q}(M)$ ($q \geq 1$) denote the standard Lebesgue and Sobolev spaces with respect to the volume measure $dv_g$, equipped with norms $\|\cdot\|_p$ and $\|\cdot\|_{1,q}$ respectively. Following convention, we write $H^1(M) = W^{1,2}(M)$ and define
\begin{align*}
    E = \left\{w \in H^1(M) \mid \int_M w dv_g = 0\right\}.
\end{align*}
Given that $h_1$ and $h_2$ are positive somewhere, the admissible space
\begin{align*}
    \mathcal{E} = \left\{w \in E \mid \int_M h_1 e^w dv_g > 0 \text{ and } \int_M h_2 e^{-w} dv_g > 0\right\}
\end{align*}
is nonempty. Note that $\mathcal{E} = E$ when $h_1, h_2 > 0$ on $M$.

Equation \eqref{eq-mfe-1} arises as the Euler-Lagrange equation for the functional
\begin{align*}
    J_{\rho_1,\rho_2}(u) = \frac{1}{2} \int_M |\nabla u|^2 dv_g - \rho_1 \log \int_M h_1 e^{u} dv_g - \rho_2 \log \int_M h_2 e^{-u} dv_g
\end{align*}
defined on $\mathcal{E}$. A fundamental tool is the Moser-Trudinger type inequality established by Ohtsuka and Suzuki \cite{OS} in 2006:
\begin{align}\label{mt-ineq}
    \log \int_M e^{u} dv_g + \log \int_M e^{-u} dv_g \leq \frac{1}{16\pi} \int_M |\nabla u|^2 dv_g + C, \quad \forall u \in E,
\end{align}
where the constant $\frac{1}{16\pi}$ is sharp. This inequality naturally leads to the following classification of \eqref{eq-mfe-1} based on the parameters $(\rho_1,\rho_2)$:

\begin{itemize}
    \item \textbf{Subcritical case}: $\rho_1, \rho_2 \in [0,8\pi)$
    \item \textbf{Critical case}:
    \begin{itemize}
        \item Partial critical case: either $\rho_1 = 8\pi$ with $\rho_2 \in [0,8\pi)$ or $\rho_2 = 8\pi$ with $\rho_1 \in [0,8\pi)$
        \item Full critical case: $\rho_1 = \rho_2 = 8\pi$
    \end{itemize}
    \item \textbf{Supercritical case}: at least one $\rho_i > 8\pi$
\end{itemize}

From the inequality \eqref{mt-ineq}, it follows that in the subcritical case, the functional $J_{\rho_1,\rho_2}$ is both bounded from below and coercive. Standard variational arguments then guarantee the attainment of its infimum.

The two partial critical cases are equivalent through the transformation $u \mapsto -u$ in \eqref{eq-mfe-1}. When $\rho_2 = 0$, equation \eqref{eq-mfe-1} reduces to
\begin{align}\label{eq-kw}
    -\Delta u = 8\pi\left(\frac{h_1e^{u}}{\int_M h_1 e^{u} dv_g} - 1\right),
\end{align}
which has been extensively studied in the literature. On the standard sphere, this becomes the celebrated Nirenberg problem. The existence theory for \eqref{eq-kw} has been developed in \cite{M71,M73,KW74,H86,CD88,CY87,CY88,H90,CL93,St05,CLLX21} and related works.

For general compact Riemann surfaces, this is known as the Kazdan-Warner problem. Let $K$ denote the Gauss curvature of $(M,g)$. Ding-Jost-Li-Wang \cite{DJLW97} pursued a variational approach and established that for strictly positive $h_1$, the condition
\begin{align}\label{cond-djlw}
    \Delta\log h_1(x) + 8\pi - 2K(x) > 0 \quad \forall x \in M,
\end{align}
guarantees the existence of a minimal type solution to \eqref{eq-kw}. This is now referred to as the Ding-Jost-Li-Wang condition. Li-Zhu \cite{LZ19}
gave a new proof based on flow methods. Subsequent work has extended this result in several directions:

\begin{itemize}
    \item For nonnegative $h_1$, Yang and the second author \cite{YZ17} showed that \eqref{cond-djlw} remains sufficient when blow-up at zeros of $h_1$ is excluded. This was later reproved by the first author and Zhu \cite{SZ21} using flow methods.

    \item When $h_1$ changes sign, the condition \eqref{cond-djlw} restricted to the positive set $M_1^+ = \{x \in M : h_1(x) > 0\}$ still ensures existence \cite{SZ24,LX22,WY22,Z24}.
\end{itemize}

In the partial critical case where $\rho_2 \in (0,8\pi)$, as well as in the full critical case, Zhou \cite{Zhou} established an existence result for constant coefficient functions. Specifically, when $h_1 = h_2 \equiv 1$, the condition
\begin{align}\label{zhou-cond-1}
(8\pi - \rho_2) - 2K(x) > 0 \quad \forall x \in M,
\end{align}
guarantees the existence of a minimal type solution to \eqref{eq-mfe-1}.

For the supercritical case, the functional $J_{\rho_1,\rho_2}$ fails to be bounded below due to the inequality \eqref{mt-ineq}, which precludes the existence of minimal type solutions. While we do not address the supercritical case in this work, we refer interested readers to \cite{Zhou,J13,B15} for developments in this direction.

The present paper focuses on the critical case of the mean field equation for equilibrium turbulence with general prescribed functions $h_1$ and $h_2$. Our primary objective is to extend Zhou's existence results to this more general setting, where the coefficient functions are no longer restricted to be constant.

In the partial critical case, we have
\begin{thm}\label{thm-1}
Let $(M,g)$ be a compact Riemann surface with unit area and Gauss curvature $K$, and let $h_1,h_2\in C^2(M)$ be functions that are positive somewhere. For $\rho_2\in(0,8\pi)$, define the positive set $M_1^+ = \{x\in M : h_1(x) > 0\}$. If the following condition holds:
\begin{align}\label{cond-1}
\Delta\log h_1(x) + (8\pi - \rho_2) - 2K(x) > 0 \quad \forall x \in M_1^+,
\end{align}
then the functional $J_{8\pi,\rho_2}$ admits a minimizer $u\in \mathcal{E}$ solving the equation:
\begin{align}\label{eq-1}
-\Delta u = 8\pi\left(\frac{h_1e^{u}}{\int_M h_1 e^u dv_g} - 1\right) - \rho_2\left(\frac{h_2e^{-u}}{\int_M h_2 e^{-u} dv_g} - 1\right).
\end{align}
\end{thm}

In the full critical case, we have
\begin{thm}\label{thm-2}
Let $(M,g)$ be a compact Riemann surface with unit area and Gauss curvature $K$, and let $h_1, h_2 \in C^2(M)$ be functions that are positive somewhere. Define the positive sets $M_i^+ = \{x \in M : h_i(x) > 0\}$ for $i = 1, 2$. If the following conditions are satisfied:
\begin{align}\label{cond-2}
\Delta \log h_i(x) - 2K(x) > 0 \quad \forall x \in M_i^+, \quad i = 1, 2,
\end{align}
then the functional $J_{8\pi,8\pi}$ admits a minimizer $u\in \mathcal{E}$ solving the equation:
\begin{align}\label{eq-2}
-\Delta u = 8\pi\left(\frac{h_1e^{u}}{\int_M h_1 e^u dv_g} - \frac{h_2e^{-u}}{\int_M h_2 e^{-u} dv_g}\right).
\end{align}
\end{thm}

The Ding-Jost-Li-Wang type condition (as seen in \eqref{cond-djlw}) arises naturally in the study of critical semilinear equations. Indeed, conditions \eqref{zhou-cond-1}, \eqref{cond-1}, and \eqref{cond-2} all represent variations of this fundamental condition. This type condition also arise in the study of existence results for critical Toda systems \cite{LL05,JW01,JLW06,SZhu24+,SZhu24++} and the critical prescribed $Q$ curvature problems \cite{LLL12,BL24}.

To end the introduction, we would like to sketch the proofs of our main theorems. For both Theorems \ref{thm-1} and \ref{thm-2}, we consider
the sequence of minimizers (denote by $u_\ep$) of the corresponding perturbed functionals. We analysis the minimizers carefully, there are two possibilities:
(1) $u_\ep$ is convergent in $H^1(M)$;
(2) $u_\ep$ or $-u_\ep$ blows up.
If we are in (1), then we are done. So we assume $u_\ep$ blows up. In the proof of Theorem \ref{thm-1}, we show that $u_\ep$  blows up only
at one point where $h_1$ is positive, $-u_\ep\leq C$  and does not blow up.
In the proof of Theorem \ref{thm-2}, there are three cases:

Case 1. $u_\ep$ only blows up at one point where
$h_1$ is positive, $-u_\ep\leq C$ and does not blow up.

Case 2. $u_\ep\leq C$ and does not blow up and $-u_\ep$ only blows up at one point where
$h_2$ is positive.

Case 3. $u_\ep$ only blows up at one point $x_1$ with $h_1(x_1)>0$ and $-u_\ep$ only blows up at one point $x_2$ with $h_2(x_2)>0$, and $x_1\neq x_2$.

Based on these facts, using blowup analysis, we know the asymptotic property of $u_\ep$ (or $-u_\ep$) near the blowup point. We can show
$u_\ep$ converges to some Green function. These combining with the Dirichlet principle can give estimate on the explicit lower bound for the functional when blowup happens. After that, we construct a sequence of test functions $\phi_\ep$ such that under assumptions \eqref{cond-1} (or \eqref{cond-2}) the value of the functional at $\phi_\ep$ is strictly smaller than the lower bound we derived. This contradiction tells us that
the blowup does not happen and we are in the first situation.

The outline of the rest of the paper is following: In Sections 2 and 3,  we present the proofs of Theorems \ref{thm-1} and \ref{thm-2} respectively. Throughout the whole paper,
the constant $C$ is varying from line to line and even in the same line, we do not distinguish
sequence and its subsequences since we only care about the existence result.

\section{Proof of Theorem \ref{thm-1}}
In this section, we shall prove Theorem \ref{thm-1}. Firstly, we do some analysis on
the partial critical case, i.e., $\rho_1=8\pi$ and $\rho_2\in(0,8\pi)$. Secondly, we derive the lower bound for $J_{8\pi,\rho_2}$ by assuming
$u_\ep$ blows up. Lastly, we construct a blowup sequence $\phi_\ep$ such that $J_{8\pi,\rho_2}(\phi_\ep)$ is smaller than the lower bound, which
contradicts $u_\ep$ blows up.
\subsection{Analysis for the partial critical case}

For any $\ep>0$ small, inequality \eqref{mt-ineq} shows that there exists a $u_\ep\in\mathcal{E}$ such that
\begin{align}\label{minimizer}
J_{8\pi-\ep,\rho_2}(u_\ep) = \inf_{\mathcal{E}} J_{8\pi-\ep,\rho_2}(u)
\end{align}
and
 \begin{align}\label{eq-uep-1}
 \begin{cases}
 -\Delta u_\ep = (8\pi-\ep)\left(\frac{h_1e^{u_\ep}}{\int_Mh_1e^{u_\ep}dv_g}-1\right) - \rho_2\left(\frac{h_2e^{-u_\ep}}{\int_Mh_2e^{-u_\ep}dv_g}-1\right),\\
 \int_Mu_\ep dv_g=0.
 \end{cases}
 \end{align}
In this subsection, we shall study analytic properties of $u_\ep$.

First of all, we have $\|\Delta u_\ep\|_1\leq C$. In fact, we have
\begin{lem}\label{lem-bd}
There exist two positive constants $C_1$ and $C_2$ such that
\begin{align*}
C_1\leq \frac{\int_Me^{u_\ep}dv_g}{\int_{M}h_1e^{u_\ep}dv_g},~\frac{\int_Me^{-u_\ep}dv_g}{\int_{M}h_2e^{-u_\ep}dv_g}\leq C_2.
\end{align*}
\end{lem}
\begin{proof}
Since $h_1$ and $h_2$ are positive somewhere, one can choose $$C_1=\min\left\{\frac{1}{\max_Mh_1},\frac{1}{\max_Mh_2}\right\}$$
such that
\begin{align}\label{bd-1}
C_1\leq \frac{\int_Me^{u_\ep}dv_g}{\int_{M}h_1e^{u_\ep}dv_g}~~\text{and}~~C_1\leq\frac{\int_Me^{-u_\ep}dv_g}{\int_{M}h_2e^{-u_\ep}dv_g}.
\end{align}
By \eqref{mt-ineq} and \eqref{minimizer}, we have
\begin{align}\label{bd-2}
&\log\int_Me^{u_\ep}dv_g+\log\int_Me^{-u_\ep}dv_g\leq\frac{1}{16\pi}\int_M|\nabla u_\ep|^2dv_g+C\nonumber\\
=&\frac{1}{8\pi}\left(J_{8\pi-\ep,\rho_2}(u_\ep)+(8\pi-\ep)\log\int_Mh_1e^{u_\ep}dv_g+\rho_2\log\int_Mh_2e^{-u_\ep}dv_g\right)\nonumber\\
\leq& C+\frac{8\pi-\ep}{8\pi}\log\int_Mh_1e^{u_\ep}dv_g+\frac{\rho_2}{8\pi}\log\int_Mh_2e^{-u_\ep}dv_g.
\end{align}
Using Jensen's inequality, we know $\log\int_Me^{u_\ep}dv_g\geq0$ and $\log\int_Me^{-u_\ep}dv_g\geq0$. Then we have by \eqref{bd-2} that
\begin{align*}
\frac{8\pi-\ep}{8\pi}\log\frac{\int_Me^{u_\ep}dv_g}{\int_Mh_1e^{u_\ep}dv_g}+\frac{\rho_2}{8\pi}\log\frac{\int_Me^{-u_\ep}dv_g}{\int_Mh_2e^{-u_\ep}dv_g}\leq C.
\end{align*}
This together with \eqref{bd-1} yields that there exists a constant $C_2$ such that
\begin{align*}
\frac{\int_Me^{u_\ep}dv_g}{\int_Mh_1e^{u_\ep}dv_g}\leq C_2~~\text{and}~~\frac{\int_Me^{-u_\ep}dv_g}{\int_Mh_2e^{-u_\ep}dv_g}\leq C_2.
\end{align*}
We finish the proof of the lemma.
\end{proof}

Secondly, the $L^s$ norm of $\nabla u_\ep$ is bounded for any $s\in(1,2)$. Namely, one has
\begin{lem}\label{lem-Ls}
For any $s\in(1,2)$, $\|\nabla u_\ep\|_{s}\leq C$.
\end{lem}
\begin{proof}
Let $s'=1/s>2$, we know by definition that
\begin{align*}
\|\nabla u_\ep\|_s=\sup\left\{\left|\int_M\nabla u_\ep\nabla \phi dv_g\right|: \phi\in W^{1,s'}(M), \int_M \phi dv_g=0, \|\phi\|_{1,s'}=1\right\}.
\end{align*}
The Sobolev embedding theorem shows that $\|\phi\|_{L^\infty(M)}\leq C$ for some constant $C$.
Meanwhile, Lemma \ref{lem-bd} tells us that $\|\Delta u_\ep\|_1$ is bounded. Then it follows from equation \eqref{eq-uep-1} that
\begin{align*}
\left|\int_M\nabla u_\ep\nabla \phi dv_g\right|=\left|\int_M\phi(-\Delta u_\ep)dv_g\right|\leq C.
\end{align*}
This ends the proof.
\end{proof}

We denote by
\begin{align*}
m_\ep=\max_M\left(u_\ep-\log\int_Mh_1e^{u_\ep}dv_g\right),~~n_\ep=\max_M\left(-u_\ep-\log\int_Mh_2e^{-u_\ep}dv_g\right).
\end{align*}

\begin{lem}\label{lem-char} The following three items are equivalent:

$(i)$ $m_{\ep}+n_{\ep}\to+\infty$ as $\ep\to0$;

$(ii)$ $\|\nabla u_\ep\|_2\to+\infty$ as $\ep\to0$;

$(iii)$ $\log\int_Mh_1e^{u_\ep}dv_g+\log\int_Mh_2e^{-u_\ep}dv_g\to+\infty$ as $\ep\to0$.
\end{lem}
\begin{proof}
$(i)\Rightarrow(ii)$: Suppose not, then $\|\nabla u_\ep\|_2\leq C$. Since $\overline{u_\ep}=0$, we know by the Moser-Trudinger inequality that
$e^{|u_\ep|}$ is bounded in $L^s(M)$ for any $s\geq1$. Using Jensen's inequality, $\int_Me^{u_\ep}dv_g\geq e^{\overline{u_\ep}}=1$ and
$\int_Me^{-u_\ep}dv_g\geq e^{-\overline{u_\ep}}=1$. This together with Lemma \ref{lem-bd} shows that
\begin{align}\label{char-1}
\frac{1}{\int_Mh_1e^{u_\ep}dv_g}\leq C_2~~\text{and}~~\frac{1}{\int_Mh_2e^{-u_\ep}dv_g}\leq C_2.
\end{align}
Taking these facts into \eqref{eq-uep-1} we have $\|\Delta u_\ep\|_2\leq C$. By Poincar\'{e}'s inequality, $\|u_\ep\|_2\leq C\|\nabla u_\ep\|_2\leq C$.
Then by the elliptic estimates, $\|u_\ep\|_{2,2}\leq C$. The Sobolev embedding theorem shows $\|u_\ep\|_{C^{1}(M)}\leq C$. It together with
\eqref{char-1} yields $m_\ep\leq C$ and $n_\ep\leq C$. This contradicts $(i)$.

$(ii)\Rightarrow(iii)$: Since $J_{8\pi-\ep,\rho_2}(u_\ep)$ is bounded, we have
\begin{align}\label{char-2}
(8\pi-\ep)\log\int_Mh_1e^{u_\ep}dv_g+\rho_2\log\int_Mh_2e^{-u_\ep}dv_g\to+\infty
\end{align}
as $\ep\to0$. Combining \eqref{char-1} and \eqref{char-2} we arrive at $(iii)$.

$(iii)\Rightarrow(i)$: If not, we have $m_\ep+n_\ep\leq C$. By Lemma \ref{lem-bd}, one has $m_\ep, n_\ep\geq\log C_1$. Then both
$m_\ep$ and $n_\ep$ are bounded. This together with \eqref{char-1} yields $\|\Delta u_\ep\|_{L^{\infty}(M)}\leq C$. By Lemma \ref{lem-Ls}
and the Sobolev inequality we obtain $\|u_\ep\|_2\leq C$. Then from the elliptic estimates we have $\|u_\ep\|_{2,2}\leq C$.  The Sobolev embedding
 theorem shows $\|u_\ep\|_{C^{1}(M)}\leq C$. Then $\max_{M}|u_\ep|\leq C$ and
\begin{align}\label{char-3}
C\geq& \max_Mu_\ep+\max_M(-u_\ep)\nonumber\\
    =&m_\ep+\log \int_Mh_1e^{u_\ep}dv_g+n_\ep+\log\int_Mh_2e^{-u_\ep}dv_g\nonumber\\
    \geq&2\log C_1+\log \int_Mh_1e^{u_\ep}dv_g+\log\int_Mh_2e^{-u_\ep}dv_g.
\end{align}
Since \eqref{char-3} contradicts $(iii)$, we have $(i)$. This finishes the proof.
\end{proof}

\begin{definition}[Blow up]\label{def} If one of the three items in Lemma \ref{lem-char} holds, we call $u_\ep$ blows up.
\end{definition}

For $(u_\ep)$, there are two possibilities: One is $\|\nabla u_\ep\|_2\leq C$, then since $\overline{u_\ep}=0$ we have by
Poincar\'{e}'s inequality that $\|u_\ep\|_{1,2}\leq C$, one can show $u_\ep\rightharpoonup u_0\in \mathcal{E}$ weakly in $H^1(M)$ as $\ep\to0$
and $u_0$ minimizes $J_{8\pi,\rho_2}$ in $\mathcal{E}$, then the proof of Theorem \ref{thm-1} can terminate. Or else,
$\|\nabla u_\ep\|_2\to+\infty$ as $\ep\to0$. By Lemma \ref{lem-Ls} and Definition \ref{def}, $u_\ep$ blows up. In the rest of the proof of
Theorem \ref{thm-1}, we always assume $u_\ep$ blows up.

\begin{lem}\label{lem-ubd}
If $u_\ep$ blows up, then
\begin{align*}
\log\int_Mh_1e^{u_\ep}dv_g\to+\infty~\text{as}~\ep\to0~\text{and}~~\log\int_Mh_2e^{-u_\ep}dv_g\leq C.
\end{align*}
\end{lem}
\begin{proof}
In view of \eqref{bd-2}, we have
\begin{align*}
 &\log\int_Mh_1e^{u_\ep}dv_g+\log\int_Mh_2e^{-u_\ep}dv_g\\
\leq&\log\max_{M}h_1+\log\max_{M}h_2+\log\int_Me^{u_\ep}dv_g+\log\int_Me^{-u_\ep}dv_g\\
\leq&\frac{8\pi-\ep}{8\pi}\log\int_Mh_1e^{u_\ep}dv_g+\frac{\rho_2}{8\pi}\log\int_Mh_2e^{-u_\ep}dv_g+C,
\end{align*}
which yields
\begin{align*}
 \frac{\ep}{8\pi}\log\int_Mh_1e^{u_\ep}dv_g+\left(1-\frac{\rho_2}{8\pi}\right)\log\int_Mh_2e^{-u_\ep}dv_g\leq C.
\end{align*}
This together with \eqref{char-1} shows $\log\int_Mh_2e^{-u_\ep}dv_g\leq C$. Since $u_\ep$ blows up, by Lemma \ref{lem-char} we know
$\log\int_Mh_1e^{u_\ep}dv_g\to+\infty$ as $\ep\to0$. This finishes the proof.
\end{proof}

By Lemma \ref{lem-Ls}, there exists a function $G\in W^{1,s}(M)$
such that $u_\ep\rightharpoonup G$ weakly in $W^{1,s}(M)$ as $\ep\to0$ for any $s\in(1,2)$.
In view of Lemma \ref{lem-bd}, there exist two nonnegative bounded measures $\mu_1$ and $\mu_2$ such that
\begin{align*}
e^{u_\ep-\log\int_Mh_1e^{u_\ep}dv_g}dv_g\to\mu_1~~\text{and}~~e^{-u_\ep-\log\int_Mh_2e^{-u_\ep}dv_g}dv_g\to\mu_2
\end{align*}
as $\ep\to0$ in the sense of measures on $M$ (In fact, the convergence here is sub-convergence, for simplicity, we do not distinguish it
in this paper). We denote by
\begin{align*}
\gamma = 8\pi h_1\mu_1 - \rho_2 h_2\mu_2
\end{align*}
and
\begin{align*}
S=\left\{x\in M:~|\gamma(\{x\})|\geq 4\pi\right\}.
\end{align*}
Using similar arguments as Lemma 2.8 in \cite{DJLW97}, we have
\begin{align}\label{uep-out}
\|u_\ep\|_{L_{\text{loc}}^{\infty}(M\setminus S)}\leq C.
\end{align}
Since $u_\ep$ blows up, $S$ is not empty. Or else, with a finite covering argument and \eqref{uep-out},
we have $\|u_\ep\|_{L^{\infty}(M)}\leq C$. Then by \eqref{char-1}, $m_\ep$ and $n_\ep$ are bounded from above. This contradicts with
$u_\ep$ blows up.
Meanwhile, by the definition of $S$, for any $x\in S$ we have
\begin{align*}
\mu_1(\{x\})\geq\frac{1}{4\max_M|h_1|}~~\text{or}~~\mu_2(\{x\})\geq\frac{2\pi}{\rho_2\max_M|h_2|}.
\end{align*}
Recall that $\mu_1$ and $\mu_2$ are bounded, so $S$ is a finite set. Without loss of generality, we assume $S=\{x_l\}_{l=1}^L$. It follows from \eqref{uep-out},
Lemma \ref{lem-ubd} and Fatou's lemma that
\begin{align}\label{mu1mu2}
\mu_1 = \sum_{l=1}^L\mu_1(\{x_l\})\delta_{x_l}~~\text{and}~~\mu_2 = \beta_2e^{-G}+\sum_{l=1}^L\mu_2(\{x_l\})\delta_{x_l},
\end{align}
where $\beta_2=\lim\limits_{\ep\to0}\frac{1}{\int_Mh_2e^{-u_\ep}dv_g}$.

\begin{lem}\label{lem-single}
If $u_\ep$ blows up, then $\emph{supp}\mu_1$ is a single point set.
\end{lem}
\begin{proof}
It follows from \eqref{mu1mu2} and Lemma \ref{lem-bd} that $\text{supp}\mu_1$ is not empty. If $\text{supp}\mu_1$ has two or more different
points, then by Aubin's inequality (cf. \cite{CL91b}, Theorem 2.1): for any $\ep'>0$, there exists a positive constant $C=C(\ep')$ such that
\begin{align*}
\log\int_Me^{u_\ep}dv_g\leq\left(\frac{1}{32\pi}+\ep'\right)\int_M|\nabla u_\ep|^2dv_g+C.
\end{align*}
Choosing $\ep'=\frac{1}{96\pi}$, using \eqref{char-1} and Lemma \ref{lem-ubd}, we have
\begin{align*}
C\geq& J_{8\pi-\ep,\rho_2}(u_\ep)\\
    =& \frac{1}{2}\int_{M}|\nabla u_\ep|^2dv_g - (8\pi-\ep)\log\int_Mh_1e^{u_\ep}dv_g-\rho_2\log\int_Mh_2e^{-u_\ep}dv_g\\
  \geq& \frac{1}{6}\int_{M}|\nabla u_\ep|^2dv_g - C,
\end{align*}
which shows $\|\nabla u_\ep\|_2\leq C$ and this contradicts with $u_\ep$ blows up. Therefore, $\text{supp}\mu_1$
is a single point set and we finish the proof.
\end{proof}

Since of \eqref{mu1mu2}, we know $\text{supp}\mu_1\subset S$. Without loss of generality, we can assume $\text{supp}\mu_1=\{x_1\}$.
Notice the definition of $\mu_1$, we have
\begin{align}\label{mu1}
h_1\mu_1=\delta_{x_1}.
\end{align}

\begin{lem}\label{lem-gamma}
For any $x_l\in S$, we have $\gamma(\{x_l\})\geq4\pi$.
\end{lem}
\begin{proof}
In view of \eqref{uep-out}, for any $x_l\in S$ and sufficiently small $r>0$, there exists some positive constant $C_0$ such that
$-u_\ep\mid_{\p B_r(x_l)}\geq -C_0$. Consider the solution of
\begin{align*}
\begin{cases}
-\Delta w_\ep = \rho_2\left(\frac{h_2e^{-u_\ep}}{\int_Mh_2e^{-u_\ep}dv_g}-1\right) - (8\pi-\ep)\left(\frac{h_1e^{u_\ep}}{\int_Mh_1e^{u_\ep}dv_g}-1\right)
~~\text{in}~~B_r(x_l),\\
w_\ep=-C_0~~\text{on}~~\p B_r(x_l).
\end{cases}
\end{align*}
By the maximum principle, we have $w_\ep\leq-u_\ep$ in $B_r(x_l)$. It follows from Lemma \ref{lem-bd} that
$\rho_2\frac{h_2e^{-u_\ep}}{\int_Mh_2e^{-u_\ep}dv_g}-(8\pi-\ep)\frac{h_1e^{u_\ep}}{\int_Mh_1e^{u_\ep}dv_g}$ is bounded in
$L^1(B_r(x_l))$, then $w_\ep\rightharpoonup w_0$ weakly in $W^{1,s}(B_r(x_l))$ as $\ep\to0$ for any $s\in(1,2)$. Here $w_0$ is the solution of
\begin{align*}
\begin{cases}
-\Delta w_0 = \rho_2h_2\beta_2e^{-G}+8\pi-\rho_2-\gamma(\{x_l\})\delta_{x_l}~~\text{in}~~B_r(x_l),\\
w_0=-C_0~~\text{on}~~\p B_r(x_l).
\end{cases}
\end{align*}
If $\gamma(\{x_l\})<0$, since of \eqref{mu1}, one has $h_2(x_l)>0$. For simplicity, we assume $r$ is
small enough to ensure $h_2(x)>0$ in $B_r(x_l)$. Hence,
\begin{align*}
\rho_2h_2\beta_2e^{-G}+8\pi-\rho_2>0~~\text{in}~~B_r(x_l)
\end{align*}
and then $-\Delta w_0 \geq -\gamma(\{x_l\})\delta_{x_l}$ in $B_r(x_l)$. By the maximum principle, we have
\begin{align*}
w_0\geq\frac{1}{2\pi}\gamma(\{x_l\})\log|x-x_l|-C ~~\text{in}~~B_r(x_l).
\end{align*}
So
\begin{align}\label{gamma-1}
e^{w_0}\geq C|x-x_l|^{\frac{\gamma({x_l})}{2\pi}}.
\end{align}
 By Lemmas \ref{lem-bd} and \ref{lem-ubd}, one obtains
\begin{align}\label{gamma-2}
\int_Me^{-u_\ep}dv_g\leq C.
\end{align}
It follows from Fatou's lemma and \eqref{gamma-2}
\begin{align}\label{gamma-3}
\int_{B_r(x_l)}e^{w_0}dv_g\leq\lim_{\ep\to0}\int_{B_r(x_l)}e^{w_\ep}dv_g\leq\lim_{\ep\to0}\int_{B_r(x_l)}e^{-u_\ep}dv_g\leq C.
\end{align}
Taking \eqref{gamma-1} into \eqref{gamma-3} we obtain that
\begin{align*}
\gamma(\{x_l\})>-4\pi,~~\forall l=1,2,\cdots,L.
\end{align*}
Recall the definition of $S$, we know $\gamma(\{x_l\})\geq4\pi$. This finishes the proof.
\end{proof}

Based on this lemma, we can further show that in the partial critical case with $\rho_1=8\pi$ and $\rho_2\in(0,8\pi)$, $-u_\ep$ will
not blow up. Precisely, there holds
\begin{lem}\label{lem--uep}
If $u_\ep$ blows up, then we have $-u_\ep\leq C$.
\end{lem}
\begin{proof}
In view of Lemma \ref{lem-gamma}, $\gamma(\{x_l\})\geq4\pi$ for any $x_l\in S$. Fixed $x_l\in S$, for sufficiently small $r>0$, there holds
\begin{align*}
-\Delta(-u_\ep)=&\rho_2\left(\frac{h_2e^{-u_\ep}}{\int_Mh_2e^{-u_\ep}dv_g}-1\right) - (8\pi-\ep)\left(\frac{h_1e^{u_\ep}}{\int_Mh_1e^{u_\ep}dv_g}-1\right)\\
\leq&(8\pi-\ep)-\rho_2~~\text{in}~~B_r(x_l).
\end{align*}
By Theorem 8.17 in \cite{GT}, we have for $s\in(1,2)$ that
\begin{align}\label{-uep-1}
\sup_{B_{r/2}(x_l)}(-u_\ep)\leq& C(\|(-u_\ep)^+\|_{L^s(B_r(x_l))}+C)\nonumber\\
\leq& C(\|u_\ep\|_{L^s(M)}+C)\nonumber\\
\leq& C(\|\nabla u_\ep\|_{L^s(M)}+C)\nonumber\\
\leq& C,
\end{align}
where in the last inequality we have used Lemma \ref{lem-Ls}.  Combining \eqref{uep-out} and \eqref{-uep-1}, we obtain $-u_\ep\leq C$
and finish the proof.
\end{proof}

For any $x_l\in S$, we have by Lemma \ref{lem--uep} that
\begin{align*}
\mu_2(\{x_l\})=\lim_{r\to0}\lim_{\ep\to0}\frac{\int_{B_r(x_l)}e^{-u_\ep}dv_g}{\int_Mh_2e^{-u_\ep}dv_g}=0.
\end{align*}
Hence, we obtain
\begin{align*}
\mu_2 = \beta_2e^{-G}.
\end{align*}

Since $\text{supp}\mu_1=\{x_1\}$ and $\mu_{2}(\{x_l\})=0$ for any $x_l\in S$, we have $S=\{x_1\}$. Recall that $h_1\mu_1=\delta_{x_1}$, there exists
$x_\ep\in M$ such that $u_\ep(x_\ep)=\max_Mu_\ep$. It is obvious that $x_\ep\to x_1$ as $\ep\to0$. Choosing an isothermal coordinate system around $x_1$,
with similar arguments like the proof of Lemma 2.5 in \cite{DJLW97}, we have
\begin{align}\label{bubble-1}
u_\ep(x_\ep+r_\ep x)-u_\ep(x_\ep)\to-2\log\left(1+\pi h_1(x_1)|x|^2\right)~~\text{in}~~C^{\infty}_{\text{loc}}(\mathbb{R}^2)
\end{align}
as $\ep\to0$, where $r_\ep=e^{-m_\ep/2}$. By Lemma \ref{lem--uep} and the Lebesgue convergence theorem, we have $\beta_2=\frac{1}{\int_Mh_2e^{-G}dv_g}$.
Concluding, one obtains $u_\ep\rightharpoonup G$ in $W^{1,s}(M)$ as $\ep\to0$ for any $s\in(1,2)$, where $G$ satisfies
\begin{align}\label{green}
\begin{cases}
-\Delta G = 8\pi(\delta_{x_1}-1)-\rho_2(\beta_2h_2e^{-G}-1)~~\text{on}~~M,\\
\int_M G dv_g = 0.
\end{cases}
\end{align}
The standard elliptic estimates together with \eqref{uep-out} yield
\begin{align}\label{uGreen}
u_\ep\to G~~\text{in}~~C_{\text{loc}}^{\infty}(\mathbb{R}^2\setminus\{x_1\})
\end{align}
as $\ep\to0$. In an isothermal coordinate system around $x_1$, one has by the elliptic estimates that
\begin{align}\label{grenn-exp}
G(x,x_1)=-4\log r&+A(x_1)+\psi(x),
\end{align}
where $r(x)=\text{dist}(x,x_1)$, $\psi$ is a smooth function which is zero at $x_1$.

\subsection{The lower bound for the partial critical case}

In this subsection, we assume $u_\ep$ blows up and derive an explicit lower bound for $J_{8\pi,\rho_2}$ when $\rho_2\in(0,8\pi)$.

Let $\delta>0$ small enough such that \eqref{grenn-exp} holds in $B_\delta(x_\ep)$. In the sequel, we denote by $o_\ep(1)$ (resp. $o_L(1)$; $o_{\delta}(1)$) the terms which tends to $0$ as $\ep\to0$ (resp. $L\to\infty$; $\delta\to0$).

Recalling that $r_\ep=e^{-m_\ep/2}$. We assume $\ep$ is small enough such that $\delta>Lr_\ep$ for any fixed $L>0$. Then
\begin{align*}
\int_{M}|\nabla u_\ep|^2dv_g = &\int_{M\setminus B_\delta(x_\ep)}|\nabla u_\ep|^2dv_g
         +\int_{B_\delta(x_\ep)\setminus B_{Lr_\ep}(x_\ep)}|\nabla u_\ep|^2dv_g\nonumber\\
         &+\int_{B_{Lr_\ep}(x_\ep)}|\nabla u_\ep|^2dv_g.
\end{align*}
By \eqref{bubble-1}, we have
\begin{align}\label{lower-2}
\int_{B_{Lr_\ep}(x_\ep)}|\nabla u_\ep|^2dv_g = 16\pi\log\left(1+\pi h_1(x_1)L^2\right)-16\pi+o_\ep(1)+o_L(1).
\end{align}
It follows from \eqref{green}, \eqref{uGreen} and \eqref{grenn-exp} that
\begin{align}\label{lower-3}
\int_{M\setminus B_\delta(x_\ep)}|\nabla u_\ep|^2dv_g =& \int_{M\setminus B_{\delta}(x_1)}|\nabla G|^2dv_g+o_\ep(1)\nonumber\\
=& -\rho_2\beta_2\int_M h_2Ge^{-G}dv_g - \int_{\partial B_{\delta}(x_1)}G\frac{\partial G}{\partial n}ds_g\nonumber\\
&+o_\delta(1)+o_\ep(1)\nonumber\\
=& -\rho_2\beta_2\int_M h_2Ge^{-G}dv_g -32\pi\log\delta+8\pi A(x_1)\nonumber\\
&+o_\delta(1)+o_\ep(1).
\end{align}
To estimate $\int_{B_{\delta}(x_\ep)\setminus B_{Lr_\ep}(x_\ep)}|\nabla u_\ep|^2dv_g$, we shall follow \cite{LL05} closely. Denote by
\begin{align*}
a_\ep = \inf_{\partial B_{Lr_\ep}(x_\ep)}u_\ep,~~~~b_\ep = \sup_{\partial B_{\delta}(x_\ep)}u_\ep.
\end{align*}
We set $a_\ep-b_\ep=u_\ep(x_\ep)+d_\ep$. Then
\begin{align*}
d_\ep = -2\log\left(1+\pi h_1(x_1)L^2\right)-\sup_{\partial B_{\delta}(x_1)}G+o_{\ep}(1).
\end{align*}
Let $f_\ep=\max\{\min\{u_\ep,a_\ep\},b_\ep\}$. We have
\begin{align*}
\int_{B_{\delta}(x_\ep)\setminus B_{Lr_\ep}(x_\ep)}|\nabla u_\ep|^2dv_g
\geq&\int_{B_{\delta}(x_\ep)\setminus B_{Lr_\ep}(x_\ep)}|\nabla f_\ep|^2dv_g\\
=&\int_{B_{\delta}(x_\ep)\setminus B_{Lr_\ep}(x_\ep)}|\nabla_{\mathbb{R}^2} f_\ep|^2dx\\
\geq&\inf_{\Psi|_{\partial B_{Lr_\ep}(0)}=a_\ep,\Psi|_{\partial B_{\delta}(0)}=b_\ep}\int_{B_{\delta}(0)\setminus B_{Lr_\ep}(0)}|\nabla_{\mathbb{R}^2} \Psi|^2dx.
\end{align*}
By the Dirichlet's principle, we know
$$\inf_{\Psi|_{\partial B_{Lr_\ep}(0)}=a_\ep,\Psi|_{\partial B_{\delta}(0)}=b_\ep}\int_{B_{\delta}(0)\setminus B_{Lr_\ep}(0)}|\nabla_{\mathbb{R}^2} \Psi|^2dx$$
is uniquely attained by the harmonic function
\begin{align*}
\begin{cases}
&-\Delta_{\mathbb{R}^2}\phi = 0,\\
&\phi|_{\partial B_{Lr_\ep}(0)}=a_\ep,\phi|_{\partial B_{\delta}(0)}=b_\ep.
\end{cases}
\end{align*}
Thus,
\begin{align*}
\phi=\frac{a_\ep-b_\ep}{-\log Lr_\ep+\log\delta}\log r - \frac{a_\ep\log\delta-b_\ep\log Lr_\ep}{-\log Lr_\ep+\log\delta},
\end{align*}
and then
\begin{align*}
\int_{B_{\delta}(0)\setminus B_{Lr_\ep}(0)}|\nabla_{\mathbb{R}^2}\phi|^2dx = \frac{4\pi(a_\ep-b_\ep)^2}{-\log(Lr_\ep)^2+\log\delta^2}.
\end{align*}
Concluding, we have
\begin{align*}
\int_{B_{\delta}(x_\ep)\setminus B_{Lr_\ep}(x_\ep)}|\nabla u_\ep|^2dv_g
\geq \frac{4\pi(a_\ep-b_\ep)^2}{-\log(Lr_\ep)^2+\log\delta^2}.
\end{align*}
Since $-\log(r_\ep)^2=m_\ep$, we obtain
\begin{align}\label{eq-neck-1}
\int_{B_{\delta}(x_\ep)\setminus B_{Lr_\ep}(x_\ep)}|\nabla u_\ep|^2dv_g
\geq& 4\pi\frac{\left(m_\ep+\log\int_Mh_1e^{u_\ep}dv_g+d_\ep\right)^2}{m_\ep-\log L^2+\log\delta^2}.
\end{align}
By Lemma \ref{lem-ubd}, one has
\begin{align}\label{eq-neck-2}
&\frac{1}{2}\int_{B_{\delta}(x_\ep)\setminus B_{Lr_\ep}(x_\ep)}|\nabla u_\ep|^2-(8\pi-\ep)\log\int_Mh_1e^{u_\ep}dv_g\nonumber\\
\leq& J_{8\pi-\ep,\rho_2}(u_\ep)+\rho_2\log\int_Mh_2e^{-u_\ep}dv_g\leq C.
\end{align}
It follows from \eqref{eq-neck-1} and \eqref{eq-neck-2} that
\begin{align}\label{eq-neck-3}
2\pi\frac{\left(m_\ep+\log\int_Mh_1e^{u_\ep}dv_g+d_\ep\right)^2}{m_\ep-\log L^2+\log\delta^2}-(8\pi-\ep)\log\int_Mh_1e^{u_\ep}dv_g\leq C.
\end{align}
Recalling that $\log\int_Mh_1e^{u_\ep}dv_g\to+\infty$ and $m_\ep\to+\infty$, we get from \eqref{eq-neck-3}
\begin{align}\label{eq-neck-4}
\frac{\log\int_Mh_1e^{u_\ep}dv_g}{m_\ep}=1+o_{\ep}(1)
\end{align}
by dividing both sides by $m_\ep$ and letting $\ep$ tend to $0$. Taking \eqref{eq-neck-4} into \eqref{eq-neck-1}, we have
\begin{align}\label{eq-neck-5}
\int_{B_{\delta}(x_\ep)\setminus B_{Lr_\ep}(x_\ep)}|\nabla u_\ep|^2dv_g
\geq&4\pi\frac{\left(m_\ep+\log\int_Mh_1e^{u_\ep}dv_g\right)^2}{m_\ep}\nonumber\\
&+16\pi\left(d_\ep+\log L^2-\log\delta^2+o_{\ep}(1)\right).
\end{align}

It follows from \eqref{lower-2}, \eqref{lower-3} and \eqref{eq-neck-5} that
\begin{align}\label{lower}
J_{8\pi-\ep,\rho_2}(u_\ep)\geq&-8\pi-8\pi\log\pi-4\pi\left(A(x_1)+2\log h_1(x_1)\right)\nonumber\\
&+2\pi\frac{\left(m_\ep-\log\int_Mh_1e^{u_\ep}dv_g\right)^2}{m_\ep}-\frac{\rho_2}{2}\beta_2\int_Mh_2Ge^{-G}dv_g\nonumber\\
&-\rho_2\log\int_Mh_2e^{-G}dv_g+o_\ep(1)+o_L(1)+o_{\delta}(1)\nonumber\\
\geq&-8\pi-8\pi\log\pi-4\pi\left(A(x_1)+2\log h_1(x_1)\right)\nonumber\\
&-\frac{\rho_2}{2}\beta_2\int_Mh_2Ge^{-G}dv_g-\rho_2\log\int_Mh_2e^{-G}dv_g\nonumber\\
&+o_\ep(1)+o_L(1)+o_{\delta}(1).
\end{align}
By letting $\ep\to0$ in \eqref{lower} first, then $L\to+\infty$ and then $\delta\to0$, we obtain finally that
\begin{align}\label{lower-bound}
\inf_{\mathcal{E}}J_{8\pi,\rho_2}(u)\geq&-8\pi-8\pi\log\pi-4\pi\left(A(x_1)+2\log h_1(x_1)\right)\nonumber\\
&-\frac{\rho_2}{2}\beta_2\int_Mh_2Ge^{-G}dv_g-\rho_2\log\int_Mh_2e^{-G}dv_g\nonumber\\
\geq&-8\pi-8\pi\log\pi-4\pi\max_{x\in M_1^+}(A(x)+2\log h_1(x))\nonumber\\
&-\frac{\rho_2}{2}\beta_2\int_Mh_2Ge^{-G}dv_g-\rho_2\log\int_Mh_2e^{-G}dv_g.
\end{align}

\subsection{Test function for the partial critical case}
In this subsection, we construct a sequence of test functions  $\phi_\ep$. We shall prove that, under assumption \eqref{cond-1},
\begin{align*}
J_{8\pi,\rho_2}(\phi_\ep)<&-8\pi-8\pi\log\pi-4\pi\max_{x\in M_1^+}(A(x)+2\log h_1(x))\nonumber\\
&-\frac{\rho_2}{2}\beta_2\int_Mh_2Ge^{-G}dv_g-\rho_2\log\int_Mh_2e^{-G}dv_g,
\end{align*}
for $\ep>0$ sufficiently small. This contradicts with \eqref{lower-bound}.

Suppose that $A(p)+2\log h_1(p)=\max_{x\in M_1^+}(A(x)+2\log h_1(x))$.
Let $(\Omega;(x^1,x^2))$ be an isothermal coordinate system around $p$ and we assume the metric to be
$$g|_\Omega = e^{\phi}\left((dx^1)^2+(dx^2)^2\right),$$
and
$$\phi=b_1(p)x^1+b_2(p)x^2+c_1(p)(x^1)^2+c_2(p)(x^2)^2+c_{12}(p)x^1x^2+O(r^3),$$
where $r(x^1,x^2)=\sqrt{(x^1)^2+(x^2)^2}$.
Moreover, we assume near $p$ that
\begin{align*}
G(x,p)=-4\log r+A(p)+&\lambda(p)x^1+\nu(p)x^2+\alpha(p)(x^1)^2+\beta(p)(x^2)^2+\xi(p)x^1x^2\\
    &+\ell(x^1,x^2)+O(r^4).
\end{align*}
It is well known that
\begin{align*}
K(p)=-(c_1(p)+c_2(p)),\\
|\nabla u|^2dv_g=|\nabla u|^2dx^1dx^2,
\end{align*}
and
\begin{align*}
\frac{\p u}{\p n}dS_g=\frac{\p u}{\p r}rd\theta.
\end{align*}
For $\alpha(p)$ and $\beta(p)$, we have the following lemma:
\begin{lem}\label{lem-alpha1}
We have
\begin{align*}
\alpha(p)+\beta(p)=4\pi-\frac{\rho_2}{2}.
\end{align*}
\end{lem}
\begin{proof}
Near $p$, we have that
\begin{align*}
2\alpha(p)+2\beta(p)+O(r)=\Delta_{\mathbb{R}^2}G(x,p)=e^{-\phi}[8\pi+\rho_2(\beta_2h_2e^{-G}-1)].
\end{align*}
By letting $r\to0$ in both sides and noticing $e^{-G(x,p)}=O(r^4)$ near $p$, we finish the proof.
\end{proof}

We set
\begin{align*}
w(x)=-2\log(1+\pi|x|^2)
\end{align*}
and define
\begin{align*}
\phi_\ep=
\begin{cases}
w(\frac{x}{\ep})+\lambda(p)r\cos\theta+\nu(p)r\sin\theta, &x\in B_{L\ep}(p),\\
G-\eta H+4\log(L\ep)-2\log(1+\pi L^2)-A(p),  &x\in B_{2L\ep}(p)\setminus B_{L\ep}(p),\\
G+4\log(L\ep)-2\log(1+\pi L^2)-A(p),&\text{otherwise}.
\end{cases}
\end{align*}
Here,
$$H=G+4\log r-A(p)-\lambda(p)r\cos\theta-\nu(p)r\sin\theta$$
and $\eta$ is a cut-off function which equals $1$ in $B_{L\ep}(p)$, equals $0$ in $M\setminus B_{2L\ep}(p)$ and satisfies $|\nabla\eta|\leq\frac{C}{L\ep}$.

We have
\begin{align}\label{energy}
\int_{M}|\nabla\phi_\ep|^2dv_g=&\int_{B_{L\ep}(p)}|\nabla\phi_\ep|^2dv_g+\int_{M\setminus B_{L\ep}(p)}|\nabla G|^2dv_g\nonumber\\
&-2\int_{M\setminus B_{L\ep}(p)}\nabla G\nabla(\eta H)dv_g+\int_{M\setminus B_{L\ep}(p)}|\nabla(\eta H)|^2dv_g.
\end{align}
By direct calculations, one obtains
\begin{align}\label{energy-1}
\int_{B_{L\ep}(p)}|\nabla\phi_\ep|^2dv_g
=&\int_{B_L(0)}|\nabla_{\mathbb{R}^2} w|^2dx^1dx^2+\pi(\lambda(p)^2+\nu(p)^2)(L\ep)^2\nonumber\\
=&16\pi\log(1+\pi L^2)-\frac{16\pi^2L^2}{1+\pi L^2}+\pi(\lambda(p)^2+\nu(p)^2)(L\ep)^2
\end{align}
and
\begin{align}\label{energy-2}
\int_{M\setminus B_{L\ep}(p)}|\nabla(\eta H)|^2dv_g=\int_{B_{2L\ep}(p)\setminus B_{L\ep}(p)}O(r^2)dv_g=O((L\ep)^4).
\end{align}
Using $\int_0^{2\pi}\ell(r,\theta)d\theta=0$ and Lemma \ref{lem-alpha1}, one has
\begin{align}\label{energy-4}
-2\int_{M\setminus B_{L\ep}(p)}\nabla G\nabla(\eta H)dv_g
=-8\pi\left(4\pi-\frac{\rho_2}{2}\right)(L\ep)^2+O((L\ep)^4).
\end{align}
By \eqref{grenn-exp}, one has
$$\int_{B_{L\ep}(p)}h_2Ge^{-G}dv_g=O((L\ep)^4).$$
This together with equation \eqref{green} yields
\begin{align}\label{energy-3}
\int_{M\setminus B_{L\ep}(p)}|\nabla G|^2dv_g
=& -\int_{M\setminus B_{L\ep}(p)}G\Delta Gdv_g-\int_{\p B_{L\ep}(p)}G\frac{\p G}{\p n}dS_g\nonumber\\
=&(8\pi-\rho_2)\int_{B_{L\ep}(p)}Gdv_g - \rho_2\beta_2\int_Mh_2Ge^{-G}dv_g\nonumber\\
 &-\int_{\p B_{L\ep}(p)}G\frac{\p G}{\p n}dS_g+O((L\ep)^4).
\end{align}
Similar as Lemma 5.2 in \cite{LL05}, since $\int_0^{2\pi}\ell(r,\theta) d\theta=\int_0^{2\pi}\frac{\p \ell}{\p r}(r,\theta)d\theta=0$,
we have by Lemma \ref{lem-alpha1} that
\begin{align}\label{energy-31}
\int_{\p B_{L\ep}(p)}G\frac{\p G}{\p n}dS_g
=& 32\pi\log(L\ep)-4\pi\left(4\pi-\frac{\rho_2}{2}\right)(L\ep)^2+\pi(\lambda(p)^2+\nu(p)^2)(L\ep)^2\nonumber\\
 &-8\pi A(p)+2\pi\left(4\pi-\frac{\rho_2}{2}\right)A(p)(L\ep)^2\nonumber\\
 &-8\pi\left(4\pi-\frac{\rho_2}{2}\right)(L\ep)^2\log(L\ep)+O((L\ep)^4\log(L\ep)).
\end{align}
It is obvious that
\begin{align}\label{energy-32}
\int_{B_{L\ep}(p)}Gdv_g = -4&\pi(L\ep)^2\log(L\ep)+2\pi(L\ep)^2\nonumber\\
&+\pi A(p)(L\ep)^2+O((L\ep)^4\log(L\ep)).
\end{align}
Substituting \eqref{energy-31} and \eqref{energy-32} into \eqref{energy-3}, we have
\begin{align}\label{energy-30}
\int_{M\setminus B_{L\ep}(p)}|\nabla G|^2dv_g
=&-32\pi\log(L\ep)+8\pi A(p)-\rho_2\beta_2\int_Mh_2Ge^{-G}dv_g\nonumber\\
 &+4\pi(8\pi-\rho_2)(L\ep)^2-\pi(\lambda(p)^2+\nu(p)^2)(L\ep)^2\nonumber\\
 &+O((L\ep)^4\log(L\ep)).
\end{align}
By taking \eqref{energy-1}, \eqref{energy-2}, \eqref{energy-4} and \eqref{energy-30} into \eqref{energy}, one obtains
\begin{align}\label{energy-last}
\int_{M}|\nabla\phi_\ep|^2dv_g=&-32\pi\log(L\ep)+8\pi A(p)-\rho_2\beta_2\int_Mh_2Ge^{-G}dv_g\nonumber\\
&+16\pi\log(1+\pi L^2)-\frac{16\pi^2L^2}{1+\pi L^2}\nonumber\\
&+O((L\ep)^4\log(L\ep)).
\end{align}
By direct calculations, we have
\begin{align}\label{mean-last}
\int_{M}\phi_\ep=&4\log(L\ep)-A(p)-2\log(1+\pi L^2)\nonumber\\
                   &-2\ep^2\log(1+\pi L^2)+O((L\ep)^4\log(L\ep)).
\end{align}
Denote by $\mathcal{M}=\frac{1}{\pi}\left(-\frac{K(p)}{2}+\frac{(b_1(p)+\lambda(p))^2+(b_2(p)+\nu(p))^2}{4}\right)$. Suppose that in $B_\delta(p)$
\begin{align*}
h_1(x)-h_1(p)=&k_1r\cos\theta+k_2r\sin\theta\\
&+k_3r^2\cos^2\theta+2k_4r^2\cos\theta\sin\theta+k_5r^2\sin^2\theta+O(r^3).
\end{align*}
Following the same calculations as Section 4.2 in \cite{SZhu24+}, we have
\begin{align}\label{loghphi-1}
 &\log\int_Mh_1e^{\phi_1^\ep}\nonumber\\
=&\log h_1(p)+\log\ep^2\nonumber\\
&+\mathcal{M}\ep^2\log(1+\pi L^2)-\left(\mathcal{M}+\frac{4\pi-\frac{\rho_2}{2}}{2\pi}\right)\ep^2\log(L\ep)^2\nonumber\\
 &+\frac{1}{2\pi h_1(p)}[k_3+k_5+k_1(b_1+\lambda)+k_2(b_2+\nu)]\ep^2\log(1+\pi L^2)\nonumber\\
 &-\frac{1}{2\pi h_1(p)}[k_3+k_5+k_1(b_1+\lambda)+k_2(b_2+\nu)]\ep^2\log(L\ep)^2\nonumber\\
 &+O(\ep^2)+O\left(\frac{1}{L^4}\right).
\end{align}
It is obvious that
\begin{align}\label{loghphi-2}
\int_Mh_2e^{-\phi_\ep}dv_g =& \log\int_Mh_2e^{-G}dv_g-4\log(L\ep)\nonumber\\
&+2\log(1+\pi L^2)+A(p)+O((L\ep)^4\log(L\ep)).
\end{align}
Taking \eqref{energy-last}, \eqref{mean-last}, \eqref{loghphi-1} and \eqref{loghphi-2} into the functional $J_{8\pi,\rho_2}$, we obtain that
\begin{align*}
J_{8\pi,\rho_2}(\phi_\ep)
=&-8\pi-8\pi\log\pi-8\pi\log h_1(p)-4\pi A(p)\nonumber\\
 &-\frac{\rho_2}{2}\beta_2\int_{M}h_2Ge^{-G}dv_g-\rho_2\log\int_Mh_2e^{-G}dv_g\nonumber\\
 &-8\pi\left[\mathcal{M}+\frac{4\pi-\frac{\rho_2}{2}}{2\pi}+\frac{k_3+k_5+k_1(b_1+\lambda)+k_2(b_2+\nu)}{2\pi h_1(p)}\right]\nonumber\\
 &~~~~\times\ep^2[\log(1+\pi L^2)-\log(L\ep)^2]\nonumber\\
 &+O(\ep^2)+O\left(\frac{1}{L^4}\right)+O((L\ep)^4\log(L\ep))+O(\ep^3\log L).
\end{align*}
Note that under the assumption \eqref{cond-1}, we have
\begin{align*}
\mathcal{N}:=&\mathcal{M}+\frac{4\pi-\frac{\rho_2}{2}}{2\pi}+\frac{k_3+k_5+k_1(b_1+\lambda)+k_2(b_2+\nu)}{2\pi h_1(p)}\nonumber\\
=&-\frac{K(p)}{2\pi}+\frac{(b_1+\lambda)^2+(b_2+\nu)^2}{4\pi}+\frac{4\pi-\frac{\rho_2}{2}}{2\pi}+\frac{\frac{1}{2}\Delta h_1(p)+k_1(b_1+\lambda)+k_2(b_2+\nu)}{2\pi h_1(p)}\nonumber\\
=&\frac{1}{4\pi}\left[\Delta\log h_1(p)+(8\pi-\rho_2)-2K(p)\right]+\frac{1}{4\pi}\left[(b_1+\lambda+k_1)^2+(b_2+\nu+k_2)^2\right]\nonumber\\
>&0,
\end{align*}
where we have used $\Delta h_1(p)=\frac{1}{2}(k_3+k_5)$ and $\nabla h_1(p)=(k_1,k_2)$.

By choosing $L^4\ep^2=\frac{1}{\log(-\log\ep)}$, we have
\begin{align*}
 J_{8\pi,\rho_2}(\phi_\ep)
=&-8\pi-8\pi\log\pi-8\pi\log h_1(p)-4\pi A(p)\nonumber\\
 &-\frac{\rho_2}{2}\beta_2\int_{M}h_2Ge^{-G}dv_g-\rho_2\log\int_Mh_2e^{-G}dv_g\nonumber\\
 &-4\pi\mathcal{N}\ep^2(-\log\ep^2)+o(\ep^2(-\log\ep^2)).
\end{align*}
Since $\mathcal{N}>0$, we have for sufficiently small $\ep$ that
\begin{align*}
J_{8\pi,\rho_2}(\phi_\ep)
<&-8\pi-8\pi\log\pi-8\pi\log h_1(p)-4\pi A(p)\nonumber\\
 &-\frac{\rho_2}{2}\beta_2\int_{M}h_2Ge^{-G}dv_g-\rho_2\log\int_Mh_2e^{-G}dv_g,
\end{align*}
which contradicts \eqref{lower-bound} and shows $u_\ep$ does not blow up. Then $u_\ep$ converges to some $u\in\mathcal{E}$ which 
minimizes $J_{8\pi,\rho_2}$ in $\mathcal{E}$ and satisfies \eqref{eq-1}. This finishes the proof of Theorem \ref{thm-1}. $\hfill{\square}$

\section{Proof of Theorem \ref{thm-2}}
In this section, we shall prove Theorem \ref{thm-2}.
Firstly, we do some analysis on
the full critical case, i.e., $\rho_1=\rho_2=8\pi$. Secondly, if $u_\ep$ blows up, then we are in three cases: since two cases can be reduced to the partial critical case,
we just need to deal with the third case (both $u_\ep$ and $-u_\ep$ blow up). We derive the lower bound for $J_{8\pi,8\pi}$ by assuming
$u_\ep$ blows up. Lastly, we construct a blowup sequence $\phi_\ep$ such that $J_{8\pi,8\pi}(\phi_\ep)$ is smaller than the lower bound, which
contradicts with $u_\ep$ blows up. We remind the reader that we still use notations like $u_\ep$ and $\phi_\ep$, but they are not the same as in Sect. 2.

\subsection{Analysis for the full critical case}

In view of inequality \eqref{mt-ineq}, for any $\ep>0$ small, there exists a $u_\ep\in \mathcal{E}$ such that
\begin{align*}
J_{8\pi-\ep,8\pi-\ep}(u_\ep) = \inf_{u\in\mathcal{E}}J_{8\pi-\ep,8\pi-\ep}(u)
\end{align*}
and
\begin{align}\label{eq-uep-2}
 \begin{cases}
 -\Delta u_\ep = (8\pi-\ep)\left(\frac{h_1e^{u_\ep}}{\int_Mh_1e^{u_\ep}dv_g}-\frac{h_2e^{-u_\ep}}{\int_Mh_2e^{-u_\ep}dv_g}\right),\\
 \int_Mu_\ep dv_g=0.
 \end{cases}
\end{align}
We shall study analytic properties of $u_\ep$ in this subsection.

The following three lemmas can be obtained with the same arguments as in Section 2, we omit the proofs and just state them below.
\begin{lem}\label{lem-bd2}
There exist two positive constants $C_1$ and $C_2$ such that
\begin{align*}
C_1\leq \frac{\int_Me^{u_\ep}dv_g}{\int_{M}h_1e^{u_\ep}dv_g}, \frac{\int_Me^{-u_\ep}dv_g}{\int_{M}h_2e^{-u_\ep}dv_g}\leq C_2.
\end{align*}
\end{lem}
\begin{lem}\label{lem-Ls2}
For any $s\in(1,2)$, $\|\nabla u_\ep\|_{s}\leq C$ for $i=1,2$.
\end{lem}
\begin{lem}\label{lem-char2} The following three items are equivalent:

$(i)$ $m_1^{\ep}+m_2^{\ep}\to+\infty$ as $\ep\to0$;

$(ii)$ $\|\nabla u_\ep\|_2\to+\infty$ as $\ep\to0$;

$(iii)$ $\log\int_Mh_1e^{u_\ep}dv_g+\log\int_Mh_2e^{-u_\ep}dv_g\to+\infty$ as $\ep\to0$.

Here,
\begin{align*}
m_1^\ep=\max_M\left(u_\ep-\log\int_Mh_1e^{u_\ep}dv_g\right),~~m_2^\ep=\max_M\left(-u_\ep-\log\int_Mh_2e^{-u_\ep}dv_g\right).
\end{align*}
\end{lem}

\begin{definition}[Blow up]\label{def2} If one of the three items in Lemma \ref{lem-char2} holds, we call $u_\ep$ blows up.
\end{definition}

For $(u_\ep)$, there are two possibilities: One is $\|\nabla u_\ep\|_2\leq C$, then since $\overline{u_\ep}=0$ we have by
Poincar\'{e}'s inequality that $\|u_\ep\|_{1,2}\leq C$, one can show $u_\ep\rightharpoonup u_0\in \mathcal{E}$ weakly in $H^1(M)$ as $\ep\to0$
and $u_0$ minimizes $J_{8\pi,8\pi}$ in $\mathcal{E}$, then the proof of Theorem \ref{thm-2} can terminate. The other is
$\|\nabla u_\ep\|_2\to+\infty$ as $\ep\to0$. By Lemma \ref{lem-Ls2} and Definition \ref{def2}, $u_\ep$ blows up. In the rest of the proof of
Theorem \ref{thm-2}, we always assume $u_\ep$ blows up.

By Lemma \ref{lem-char2} (iii), we divide the whole proof into three cases:

\textbf{Case 1} $\log\int_Mh_1e^{u_\ep}dv_g\to+\infty$, $\log\int_Mh_2e^{-u_\ep}dv_g\leq C$ as $\ep\to0$;

\textbf{Case 2} $\log\int_Mh_1e^{u_\ep}dv_g\leq C$, $\log\int_Mh_2e^{-u_\ep}dv_g\to+\infty$ as $\ep\to0$;

\textbf{Case 3} $\log\int_Mh_1e^{u_\ep}dv_g\to+\infty$, $\log\int_Mh_2e^{-u_\ep}dv_g\to+\infty$ as $\ep\to0$.

Suppose we are in Case 1, by checking the proof of Theorem \ref{thm-1} carefully, we find that $\rho_2<8\pi$ is used to show
$\log\int_Mh_2e^{-u_\ep}dv_g\leq C$ which happens to be the situation in Case 1. At any other place $\rho_2<8\pi$
can be replaced by $\rho_2=8\pi$. By Theorem \ref{thm-1}, if
$$\Delta \log h_1(x)-2K(x)>0~~\text{for~any}~x\in M_1^+,$$
where $M_1^+=\{x\in M:~h_1(x)>0\}$, $J_{8\pi,8\pi}$ has a minimizer $u\in\mathcal{E}$ which satisfies \eqref{eq-2}.

Suppose we are in Case 2, similar as Case 1, if
$$\Delta \log h_2(x)-2K(x)>0~~\text{for~any}~x\in M_2^+,$$
where $M_2^+=\{x\in M:~h_2(x)>0\}$, then $J_{8\pi,8\pi}$ has a minimizer $u\in\mathcal{E}$ which satisfies \eqref{eq-2}.

If we are in Case 3, by Lemma \ref{lem-Ls2}, there exists $G$ such that $u_\ep\rightharpoonup G$ weakly in $W^{1,s}(M)$
as $\ep\to0$ for any $s\in(1,2)$.
In view of Lemma \ref{lem-bd2}, there exist two nonnegative bounded measures $\mu_1$ and $\mu_2$ such that
\begin{align*}
e^{u_\ep-\log\int_Mh_1e^{u_\ep}dv_g}dv_g\to\mu_1~~\text{and}~~e^{-u_\ep-\log\int_Mh_2e^{-u_\ep}dv_g}dv_g\to\mu_2
\end{align*}
as $\ep\to0$ in the sense of measures on $M$. We denote by
\begin{align*}
\gamma = 8\pi h_1\mu_1 - 8\pi h_2\mu_2
\end{align*}
and
\begin{align*}
S=\{x\in M:~|\gamma(\{x\})|\geq 4\pi\}.
\end{align*}
Using similar arguments as Lemma 2.8 in \cite{DJLW97}, we have
\begin{align}\label{uep-out2}
\|u_\ep\|_{L_{\text{loc}}^{\infty}(M\setminus S)}\leq C.
\end{align}

Since $u_\ep$ blows up, $S$ is not empty. Or else, following a finite covering argument and \eqref{uep-out2}, we have
$\|u_\ep\|_{L^\infty(M)} \leq C$, which contradicts Case 3. By the definition of $S$, for any $x \in S$, there holds
$$\mu_1(\{x\}) \geq \frac{1}{4\max_M |h_1|} \quad \text{or} \quad \mu_2(\{x\}) \geq \frac{1}{4\max_M |h_2|}.$$
In view of $\mu_1$ and $\mu_2$ are bounded, $S$ is a finite set. We denote by $S=\{x_l\}_{l=1}^L$.
It follows from
\eqref{uep-out2} and Fatou's lemma that
\begin{align}\label{mu_i=}
\mu_i=\sum_{l=1}^L\mu_i(\{x_l\})\delta_{x_l},~~i=1,2,
\end{align}
where $\delta_{x}$ is the Dirac distribution.

By Lemma \ref{lem-bd2} and \eqref{mu_i=} one knows that ${\rm{supp}}\mu_i\neq\emptyset$, $i=1,2$. If there are at least two points in each ${\rm{supp}}\mu_i$, then by the improved Moser-Trudinger inequality (cf. \cite[Proposition 2.6]{J13}), for any $\ep'>0$, there exists some $C=C(\ep')>0$ such that
\begin{align*}
\log\int_Me^{u_\ep}+\log\int_Me^{-u_\ep}&\leq\left(\frac{1}{32\pi}+\ep'\right)\int_M |\nabla u_\ep|^2dv_g+C.
\end{align*}
By choosing $\ep'=\frac{1}{96\pi}$, we have
\begin{align*}
C\geq& J_{8\pi-\ep,8\pi-\ep}(u_\ep)\nonumber\\
 =&\frac{1}{2}\int_M|\nabla u_\ep|^2dv_g-(8\pi-\ep)\log\int_Mh_1e^{u_\ep}dv_g-(8\pi-\ep)\log\int_Mh_2e^{-u_\ep}dv_g\nonumber\\
 \geq&\frac{1}{6}\int_M|\nabla u_\ep|^2dv_g-(8\pi-\ep)(\log\max_Mh_1+\log\max_Mh_2)-C,
\end{align*}
which shows $\|\nabla u_\ep\|_2$ is bounded and contradicts with $u_\ep$ blows up (Lemma \ref{lem-char2} (ii)).
Hence, either ${\rm{supp}}\mu_1$ or ${\rm{supp}}\mu_2$ has only one point. Without loss of generality, we assume that ${\rm{supp}}\mu_1$ has only one point and ${\rm{supp}}\mu_1=\{x_1\}$ since ${\rm{supp}}\mu_1\subset S$. Recalling the definition of $\mu_1$, we have
\begin{align}\label{h_1mu_1-2}
h_1\mu_1=\delta_{x_1}.
\end{align}

The following result which is based on Pohozaev identities is very important in the understanding of blowup set.
\begin{lem}\label{lem-poho}
Denote by $h_i\mu_i=\sigma_i$ for $i=1,2$, we have
\begin{align}\label{pohozaev8}
\left[\sigma_1(\{x_l\})-\sigma_2(\{x_l\})\right]^2=\sigma_1(\{x_l\})+\sigma_2(\{x_l\}),
\end{align}
where $l=1,2,\cdots,L$.
\end{lem}
\begin{proof}
It follows from \eqref{uep-out2}, \eqref{mu_i=} and \eqref{h_1mu_1-2} that $G$ satisfies
\begin{align}\label{eq-green-2}
\begin{cases}
-\Delta G = 8\pi(\delta_{x_1}-h_2\sum_{l=1}^L\mu_2(\{x_l\})\delta_{x_l}),\\
\int_MGdv_g=0.
\end{cases}
\end{align}
By the standard elliptic estimates, we have
\begin{align}\label{converge-out}
u_\ep\to G~~\text{in}~~C^2_{\text{loc}}(M\setminus S)
\end{align}
as $\ep\to0$. Inspired by the proof of Estimation B in \cite{CLin},
We let $(B_{\delta}(x_l);(x^1,x^2))$ be a coordinate system around $x_l$ and we assume the metric to be
$$g|_\Omega = e^{\phi}\left((dx^1)^2+(dx^2)^2\right)$$
with $\phi(0)=0$ and $\nabla_{\mathbb{R}^2}\phi(0)=0$. Here $\nabla_{\mathbb{R}^2}=(\frac{\p}{\partial x^1},\frac{\p}{\partial x^2})$.
It is well known that
\begin{align}\label{pohozaev0}
G = -\frac{\gamma(\{x_l\})}{2\pi}\log r + \psi,
\end{align}
where $r=\sqrt{(x^1)^2+(x^2)^2}$ and $\psi$ is a smooth function near $x_l$. In this coordinate system, \eqref{eq-uep-2} reduces to
\begin{align}\label{pohozaev1}
-\Delta_{\mathbb{R}^2}u_\ep = (8\pi-\ep)e^{\phi}\left(\frac{h_1e^{u_\ep}}{\int_Mh_1e^{u_\ep}dv_g}-\frac{h_2e^{-u_\ep}}{\int_Mh_2e^{-u_\ep}dv_g}\right)
\end{align}
for $|x|\leq\delta$, where $\Delta_{\mathbb{R}^2}=\frac{\p^2}{\p(x^1)^2}+\frac{\p^2}{\p(x^2)^2}$ is the Laplacian in $\mathbb{R}^2$.
By \eqref{pohozaev1} we know $u_\ep$ satisfies
\begin{align}\label{pohozaev2}
-\Delta_{\mathbb{R}^2}u_\ep = (8\pi-\ep)\hat{h}_1e^{u_\ep}-(8\pi-\ep)\hat{h}_2e^{-u_\ep}
\end{align}
for $|x|\leq\delta$, where
\begin{align}\label{pohozaev3}
\hat{h}_1(x) = e^{\phi(x)}\frac{h_1(x)}{\int_Mh_1e^{u_\ep}dv_g},~\hat{h}_2(x) = e^{\phi(x)}\frac{h_2(x)}{\int_Mh_2e^{-u_\ep}dv_g}.
\end{align}
It follows from the choice of $\phi(x)$ and \eqref{pohozaev3} that
\begin{align}\label{pohozaev4}
\hat{h}_1(0) = \frac{h_1(x_l)}{\int_Mh_1e^{u_\ep}dv_g}~~\text{and}~~\nabla_{\mathbb{R}^2} \hat{h}_1(0) = \frac{\nabla h_1(x_l)}{\int_Mh_1e^{u_\ep}dv_g},\nonumber\\
\hat{h}_2(0) =\frac{h_2(x_l)}{\int_Mh_2e^{-u_\ep}dv_g}~~\text{and}~~\nabla_{\mathbb{R}^2} \hat{h}_2(0) = \frac{\nabla h_2(x_l)}{\int_Mh_2e^{-u_\ep}dv_g}.
\end{align}
Multiplying \eqref{pohozaev2} with $x\cdot\nabla_{\mathbb{R}^2}u_\ep$ and integrating on $\mathbb{B}_{\delta}(0)$,
we have the following Pohozaev identity
\begin{align}\label{pohozaev5}
&-\delta\int_{\p\mathbb{B}_{\delta}(0)}\left(\left(\frac{\p u_\ep}{\p r}\right)^2-\frac{1}{2}|\nabla_{\mathbb{R}^2}u_\ep|^2\right)ds\nonumber\\
=& (8\pi-\ep)\delta\int_{\p\mathbb{B}_{\delta}(0)}\hat{h}_1e^{u_\ep}ds
  - (8\pi-\ep)\int_{\mathbb{B}_{\delta}(0)}(2\hat{h}_1e^{u_\ep}+x\cdot\nabla_{\mathbb{R}^2}\hat{h}_1e^{u_\ep})dx\nonumber\\
  &+ (8\pi-\ep)\delta\int_{\p\mathbb{B}_{\delta}(0)}\hat{h}_2e^{-u_\ep}ds
  - (8\pi-\ep)\int_{\mathbb{B}_{\delta}(0)}(2\hat{h}_2e^{-u_\ep}+x\cdot\nabla_{\mathbb{R}^2}\hat{h}_2e^{-u_\ep})dx.
\end{align}
Letting $\ep\to0$ first and then $\delta\to0$ in \eqref{pohozaev5}, by using \eqref{converge-out}, \eqref{pohozaev0}, \eqref{pohozaev3} and \eqref{pohozaev4} we conclude
\begin{align}\label{pohozaev7}
-\pi\left(-\frac{\gamma(\{x_l\})}{2\pi}\right)^2=-16\pi\left[h_1(x_l)\mu_1(\{x_l\})+h_2(x_l)\mu_2(\{x_l\})\right].
\end{align}
Recalling that $h_i\mu_i=\sigma_i$ for $i=1,2$, then \eqref{pohozaev7} reduces to \eqref{pohozaev8}, this ends the proof.
\end{proof}

We know from \eqref{h_1mu_1-2} that $\sigma_1(\{x_1\})=1$ and $\sigma_1(\{x_l\})=0$ for any $l\geq2$, taking this fact into \eqref{pohozaev8} we obtain that
\begin{align*}
&\sigma_2(\{x_1\})=0~~\text{or}~~\sigma_2(\{x_1\})=3;\\
&\sigma_2(\{x_l\})=0~~\text{or}~~\sigma_2(\{x_l\})=1,~~\forall l\geq2.
\end{align*}
This together with $\sigma_2(M)=1$ and \eqref{mu_i=} yields
\begin{align*}
\sigma_2(\{x_m\})=1~~\text{for~some}~m\geq2~~\text{and}~\sigma_2(\{x_l\})=0~~\forall l\in\{1,\cdots,L\}\setminus\{m\}.
\end{align*}
Without loss of generality, we assume $m=2$. Then we have
\begin{align}\label{h_2mu_2}
h_2\mu_2=\sigma_2=\delta_{x_2}.
\end{align}
We would like to collect \eqref{h_1mu_1-2} and \eqref{h_2mu_2} as the following lemma.
\begin{lem}\label{mu_1mu_2}
If we are in Case 3, that is, $\log\int_Mh_1e^{u_\ep}dv_g\to+\infty$ and $\log\int_Mh_2e^{-u_\ep}dv_g\to+\infty$ as $\ep\to0$.
Then $h_1\mu_1=\delta_{x_1}$ and $h_2\mu_2=\delta_{x_2}$ with $x_1\neq x_2$.
\end{lem}

\begin{rem} From the definition of $S$, we know $S=\{x_1,x_2\}$.
\end{rem}

To do blow-up analysis near $x_1$ (resp. $x_2$), one still needs the upper bound of $-u^\ep$ (resp. $u_\ep$) near $x_1$ (resp. $x_2$).
In fact, we have
\begin{lem}\label{upperbound}
There exists  a positive number $r$ which is less than $\text{dist}(x_1,x_2)/2$ and makes $h_i>0$ in $B_r(x_i)$ for $i=1,2$, such that
\begin{align*}
\sup_{B_{r/4}(x_1)}\left(-u_\ep\right) \leq C~~\text{and}~~\sup_{B_{r/4}(x_2)} u_\ep\leq C.
\end{align*}
\end{lem}
\begin{proof} For $i=1$, we consider the solution of
\begin{align*}
\begin{cases}
-\Delta v^{1}_\ep = (8\pi-\ep)\frac{h_2e^{-u_\ep}}{\int_Mh_2e^{-u_\ep}dv_g} &\text{in}~B_r(x_1),\\
v^{1}_\ep=0 &\text{on}~\p B_r(x_1).
\end{cases}
\end{align*}
Denote by $v^2_\ep=-u_\ep-v^1_\ep$, then
\begin{align*}
-\Delta v^2_\ep = -(8\pi-\ep)\frac{h_1e^{u_\ep}}{\int_Mh_1e^{u_\ep}dv_g}\leq 0~~\text{in}~B_r(x_1),
\end{align*}
since $h_1>0$ in $B_r(x_1)$. By Theorem 8.17 in \cite{GT} and Lemma \ref{lem-Ls2}, we have for $s\in(1,2)$ that
\begin{align*}
\sup_{B_{r/2}(x_1)}v^2_\ep \leq& C\left(\|(v^2_\ep)^+\|_{L^s(B_r(x_1))}+C\right)\\
\leq& C\left(\|-u_\ep\|_{L^s(M)}+\|v^1_\ep\|_{L^s(B_r(x_1))}+C\right)\\
\leq& C\left(\|\nabla u_\ep\|_{L^s(M)}+\|v^1_\ep\|_{L^s(B_r(x_1))}+C\right)\\
\leq& C\left(\|v^1_\ep\|_{L^s(B_r(x_1))}+C\right).
\end{align*}
Since
$$\lim_{r\to0}\lim_{\ep\to0}\frac{\int_{B_r(x_1)}h_2e^{-u_\ep}dv_g}{\int_Mh_2e^{-u_\ep}dv_g}=\mu_2(\{x_1\})=0,$$
it follows from Theorem 1 in \cite{BM} that: for sufficiently small $r$,  $\int_{B_r(x_1)}e^{t|v^1_\ep|}dv_g\leq C$ for some $t>1$, which yields that
$$\|v^1_\ep\|_{L^s(B_r(x_1))}\leq C.$$
Then we have
$$\sup_{B_{r/2}(x_1)}v^2_\ep\leq C.$$
Note that
\begin{align*}
\int_{B_{r/2}(x_1)}e^{-tu_\ep}dv_g=&\int_{B_{r/2}(x_1)}e^{tv^2_\ep}e^{tv^1_\ep}dv_g\\
\leq& C\int_{B_{r/2}(x_1)}e^{t|v^1_\ep|}dv_g\\
\leq& C.
\end{align*}
By the standard elliptic estimates, we have
\begin{align*}
\|v^1_\ep\|_{L^{\infty}(B_{r/4}(x_1))} \leq C.
\end{align*}
Therefore, one obtains that
\begin{align*}
-u_\ep\leq C~~\text{in}~~B_{r/4}(x_1).
\end{align*}
Similarly, we can prove
\begin{align*}
u_\ep\leq C~~\text{in}~~B_{r/4}(x_2).
\end{align*}
This finishes the proof.
\end{proof}

Assume $\max_{M}u_\ep(x)=u_\ep(x_1^\ep)$, $\max_{M}(-u_\ep)(x)=-u_\ep(x_2^\ep)$, it follows from \eqref{uep-out2}, Lemmas \ref{mu_1mu_2} and \ref{upperbound} that
\begin{align*}
x_i^\ep\to x_i~~\text{as}~~\ep\to0,~i=1,2.
\end{align*}
Let $(\Omega_i;(x^1,x^2))$ be an isothermal coordinate system around $x_i$ and we assume the metric to be
$$g|_{\Omega_i} = e^{\phi_i}\left((dx^1)^2+(dx^2)^2\right),~~\phi_i(0)=0.$$
Similar as Case 1 in \cite{LL05} and Lemma 2.5 in \cite{DJLW97}, we have as $\ep\to0$ that
\begin{align}\label{bubble-2}
\begin{cases}
u_\ep(x^\ep_1+r_1^{\ep}x)-u_\ep(x^\ep_1)\to -2\log\left(1+\pi h_1(x_1)|x|^2\right),\\
-u_\ep(x^\ep_2+r_2^\ep x)+u_\ep(x^\ep_2)\to -2\log(1+\pi h_2(x_2)|x|^2),
\end{cases}
\end{align}
where $r_i^\ep=e^{-m_i^\ep/2}$ for $i=1,2$ with $m_1^\ep=u_\ep(x_1^\ep)-\log\int_Mh_1e^{u_\ep}dv_g$ and  $m_2^\ep=-u_\ep(x_2^\ep)-\log\int_Mh_2e^{-u_\ep}dv_g$.

By taking \eqref{h_2mu_2} into \eqref{eq-green-2}, we have
\begin{align}\label{green-2}
\begin{cases}
-\Delta G = 8\pi(\delta_{x_1}-\delta_{x_2}),\\
\int_MGdv_g=0.
\end{cases}
\end{align}
Recalling that for any $s\in(1,2)$, we have
$u_\ep$ converges to $G$ weakly in $W^{1,s}(M)$ and strongly in $C^2_{\text{loc}}(M\setminus\{x_1,x_2\})$ as $\ep\to0$.
Therefore, in $\Omega_1$,
\begin{align}\label{green2-1}
G(x) = -4\log r + A_{x_1}+f_1,
\end{align}
where $f_1$ is a smooth function which is zero at $x_1$. In $\Omega_2$,
\begin{align}\label{green2-2}
G(x) = 4\log r + A_{x_2}+ f_2,
\end{align}
where $f_2$ is a smooth function which is zero at $x_2$.

\subsection{The lower bound for the full critical case}

In this subsection, by assume $u_\ep$ blows up and Case 3 happens, we shall derive an explicit lower bound for $J_{8\pi,8\pi}$.

Let $\delta>0$ small enough such that \eqref{green2-1} holds in $B_\delta(x_1^\ep)$ and \eqref{green2-2} holds in $B_\delta(x_2^\ep)$.
Recalling that $r_i^\ep=e^{-m_i^\ep/2}$, we assume $\ep$ is small enough such that $\delta>Lr_i^\ep$ ($i=1,2$)
for any fixed $L>0$. Then, one has
\begin{align*}
\int_{M}|\nabla u_\ep|^2dv_g = &\int_{B_{Lr_1^\ep}(x_1^\ep)}|\nabla u_\ep|^2dv_g+\int_{B_{Lr_2^\ep}(x_2^\ep)}|\nabla u_\ep|^2dv_g\nonumber\\
         &+\int_{B_\delta(x_1^\ep)\setminus B_{Lr_1^\ep}(x_1^\ep)}|\nabla u_\ep|^2dv_g
          +\int_{B_\delta(x_2^\ep)\setminus B_{Lr_2^\ep}(x_2^\ep)}|\nabla u_\ep|^2dv_g\nonumber\\
         &+\int_{M\setminus (B_\delta(x_1^\ep)\cup B_\delta(x_2^\ep))}|\nabla u_\ep|^2dv_g.
\end{align*}
By \eqref{bubble-2}, we have
\begin{align}\label{lower2-2}
\int_{B_{Lr_1^\ep}(x_1^\ep)}|\nabla u_\ep|^2dv_g = 16\pi\log\left(1+\pi h_1(x_1)L^2\right)-16\pi+o_\ep(1)+o_R(1)
\end{align}
and
\begin{align}\label{lower2-3}
\int_{B_{Lr_2^\ep}(x_2^\ep)}|\nabla u_\ep|^2dv_g = 16\pi\log(1+\pi h_2(x_2)L^2)-16\pi+o_\ep(1)+o_R(1).
\end{align}
By \eqref{green2-1} and \eqref{green2-2}, we obtain
\begin{align}\label{lower2-4}
&\int_{M\setminus (B_\delta(x_1^\ep)\cup B_\delta(x_2^\ep))}|\nabla u_\ep|^2dv_g\nonumber\\
=& \int_{M\setminus (B_{\delta}(x_1)\cup B_{\delta}(x_2))}|\nabla G|^2dv_g+o_\ep(1)\nonumber\\
=&  - \int_{\partial B_{\delta}(x_1)}G\frac{\partial G}{\partial n}ds_g
 - \int_{\partial B_{\delta}(x_2)}G\frac{\partial G}{\partial n}ds_g+o_\delta(1)+o_\ep(1)\nonumber\\
=& -64\pi\log\delta+8\pi (A_{x_1}-A_{x_2})+o_\delta(1)+o_\ep(1).
\end{align}
To estimate $\int_{B_{\delta}(x_i^\ep)\setminus B_{Lr_i^\ep}(x_i^\ep)}|\nabla u_\ep|^2dv_g$ ($i=1,2$),
we shall follow \cite{LL05} closely. Denote by
\begin{align*}
a_1^\ep = \inf_{\partial B_{Lr_1^\ep}(x_1^\ep)}u_\ep,~~~~b_1^\ep = \sup_{\partial B_{\delta}(x_1^\ep)}u_\ep,\\
a_2^\ep = \inf_{\partial B_{Lr_2^\ep}(x_2^\ep)}(-u_\ep),~~~~b_2^\ep = \sup_{\partial B_{\delta}(x_2^\ep)}(-u_\ep).
\end{align*}
We set $a_1^\ep-b_1^\ep=u_\ep(x_1^\ep)+d_1^\ep$ and $a_2^\ep-b_2^\ep=-u_\ep(x_2^\ep)+d_2^\ep$. Then
\begin{align*}
d_1^\ep = -2\log\left(1+\pi h_1(x_1)L^2\right)-\sup_{\partial B_{\delta}(x_1)}G+o_{\ep}(1),\\
d_2^\ep = -2\log(1+\pi h_2(x_2)L^2)-\sup_{\p B_{\delta}(x_2)}(-G)+o_\ep(1).
\end{align*}
Define $f_1^\ep=\max\{\min\{u_\ep,a_1^\ep\},b_1^\ep\}$ and $f_2^\ep=\max\{\min\{-u_\ep,a_2^\ep\},b_2^\ep\}$. We have
\begin{align*}
\int_{B_{\delta}(x_i^\ep)\setminus B_{Lr_i^\ep}(x_i^\ep)}|\nabla u_\ep|^2dv_g
\geq&\int_{B_{\delta}(x_i^\ep)\setminus B_{Lr_i^\ep}(x_i^\ep)}|\nabla f_i^\ep|^2dv_g\\
=&\int_{B_{\delta}(x_i^\ep)\setminus B_{Lr_i^\ep}(x_i^\ep)}|\nabla_{\mathbb{R}^2} f_i^\ep|^2dx\\
\geq&\inf_{\Psi_i|_{\partial B_{Lr_i^\ep}(0)}=a_i^\ep,\Psi_i|_{\partial B_{\delta}(0)}=b_i^\ep}\int_{B_{\delta}(0)\setminus B_{Lr_i^\ep}(0)}|\nabla_{\mathbb{R}^2} \Psi_i|^2dx.
\end{align*}
By the Dirichlet's principle, we know
$$\inf_{\Psi_i|_{\partial B_{Lr_i^\ep}(0)}=a_i^\ep,\Psi|_{\partial B_{\delta}(0)}=b_i^\ep}\int_{B_{\delta}(0)\setminus B_{Lr_i^\ep}(0)}|\nabla_{\mathbb{R}^2} \Psi_i|^2dx$$
is uniquely attained by the harmonic function
\begin{align*}
\begin{cases}
&-\Delta_{\mathbb{R}^2}\phi_i = 0~~\text{in}~~B_{\delta}(0)\setminus B_{Lr_i^\ep}(0),\\
&\phi_i|_{\partial B_{Lr_i^\ep}(0)}=a_i^\ep,\phi_i|_{\partial B_{\delta}(0)}=b_i^\ep.
\end{cases}
\end{align*}
Thus,
\begin{align*}
\phi_i=\frac{a_i^\ep-b_i^\ep}{-\log Lr_i^\ep+\log\delta}\log r - \frac{a_i^\ep\log\delta-b_i^\ep\log Lr_i^\ep}{-\log Lr_i^\ep+\log\delta},
\end{align*}
and then
\begin{align*}
\int_{B_{\delta}(0)\setminus B_{Lr_i^\ep}(0)}|\nabla_{\mathbb{R}^2}\phi_i|^2dx = \frac{4\pi(a_i^\ep-b_i^\ep)^2}{-\log(Lr_i^\ep)^2+\log\delta^2}.
\end{align*}
Concluding, we have
\begin{align*}
\int_{B_{\delta}(x_i^\ep)\setminus B_{Lr_i^\ep}(x_i^\ep)}|\nabla u_\ep|^2dv_g
\geq \frac{4\pi(a_i^\ep-b_i^\ep)^2}{-\log(Lr_i^\ep)^2+\log\delta^2}.
\end{align*}
Since $-\log(r_i^\ep)^2=m_i^\ep$, we obtain
\begin{align}\label{eq2-neck-1}
\int_{B_{\delta}(x_i^\ep)\setminus B_{Lr_i^\ep}(x_i^\ep)}|\nabla u_\ep|^2dv_g
\geq& 4\pi\frac{(m_i^\ep+\lambda_i^\ep+d_\ep)^2}{m_i^\ep-\log L^2+\log\delta^2},
\end{align}
where $\lambda_1^\ep=\log\int_Mh_1e^{u_\ep}dv_g$ and $\lambda_2^\ep=\log\int_Mh_2e^{-u_\ep}dv_g$.
One has
\begin{align}\label{eq2-neck-2}
&\frac{1}{2}\left(\int_{B_{\delta}(x_1^\ep)\setminus B_{Lr_1^\ep}(x_1^\ep)}|\nabla u_\ep|^2dv_g
+\int_{B_{\delta}(x_2^\ep)\setminus B_{Lr_2^\ep}(x_2^\ep)}|\nabla u_\ep|^2dv_g\right)\nonumber\\
&\,\,\,\,-(8\pi-\ep)\log\int_Mh_1e^{u_\ep}dv_g-(8\pi-\ep)\log\int_Mh_2e^{-u_\ep}dv_g\nonumber\\
\leq& J_{8\pi-\ep,8\pi-\ep}(u_\ep)\leq C.
\end{align}
It follows form \eqref{eq2-neck-1} and \eqref{eq2-neck-2} that
\begin{align}\label{eq2-neck-3}
2\pi\left(\frac{(m_1^\ep+\lambda_1^\ep+d_1^\ep)^2}{m_1^\ep-\log L^2+\log\delta^2}
+\frac{(m_2^\ep+\lambda_2^\ep+d_2^\ep)^2}{m_2^\ep-\log L^2+\log\delta^2}\right)
-(8\pi-\ep)(\lambda_1^\ep+\lambda_2^\ep)\leq C.
\end{align}
Recalling that $\lambda_i^\ep\to+\infty$ and $m_i^\ep\to+\infty$ as $\ep\to0$ for $i=1,2$, we get from \eqref{eq2-neck-3}
\begin{align}\label{eq2-neck-4}
\frac{\lambda_i^\ep}{m_i^\ep}=1+o_{\ep}(1).
\end{align}
Taking \eqref{eq2-neck-4} into \eqref{eq2-neck-1}, we have
\begin{align}\label{eq2-neck-5}
\int_{B_{\delta}(x_i^\ep)\setminus B_{Lr_i^\ep}(x_i^\ep)}|\nabla u_\ep|^2dv_g
\geq&4\pi\frac{(m_i^\ep+\lambda_i^\ep)^2}{m_i^\ep}\nonumber\\
&+16\pi\left(d_i^\ep+\log L^2-\log\delta^2+o_{\ep}(1)\right).
\end{align}

It follows from \eqref{lower2-2}, \eqref{lower2-3}, \eqref{lower2-4} and \eqref{eq2-neck-5} that
\begin{align}\label{lower2}
J_{8\pi-\ep,8\pi-\ep}(u_\ep)\geq&-16\pi-16\pi\log\pi-4\pi\left(2\log h_1(x_1)+A_{x_1}\right)\nonumber\\
&-4\pi\left(2\log h_2(x_2)-A_{x_2}\right)+o_\ep(1)+o_L(1)+o_{\delta}(1).
\end{align}
By letting $\ep\to0$ in \eqref{lower2} first, then $L\to+\infty$ and then $\delta\to0$, we obtain finally that
\begin{align}\label{lower2-bound}
\inf_{\mathcal{E}}J_{8\pi,8\pi}(u)\geq&-16\pi-16\pi\log\pi-4\pi\left(2\log h_1(x_1)+A_{x_1}\right)\nonumber\\
&-4\pi\left(2\log h_2(x_2)-A_{x_2}\right).
\end{align}

\subsection{Test function for the full critical case}
In this subsection, we construct a sequence of test functions (still denote by $\phi_\ep$). We shall prove that, under assumption \eqref{cond-2},
\begin{align*}
J_{8\pi,8\pi}(\phi_\ep)<&-16\pi-16\pi\log\pi-4\pi\left(2\log h_1(x_1)+A_{x_1}\right)\nonumber\\
&-4\pi\left(2\log h_2(x_2)-A_{x_2}\right),
\end{align*}
for $\ep>0$ sufficiently small. This contradicts with \eqref{lower2-bound}.

Let $(\Omega_i;(x,y))$ be an isothermal coordinate system around $x_i$ for $i=1,2$ and we assume the metric to be
$$g|_{\Omega_i} = e^{\phi_i}(dx^2+dy^2),$$
and
$$\phi_i=b_1(x_i)x+b_2(x_i)y+c_1(x_i)x^2+c_2(x_i)y^2+c_{12}(x_i)xy+O(r^3),$$
where $r(x,y)=\sqrt{(x)^2+(y)^2}$.
Moreover we assume near $x_i$ that
\begin{align*}
G=a_i\log r+A_{x_i}+&\lambda_ix+\nu_iy+\alpha_ix^2+\beta_iy^2\nonumber\\&+\xi_ixy+\tau_i(x,y)+O(r^4), ~i=1,2,
\end{align*}
where $a_1=-4,~a_2=4$. It is well known that
\begin{align*}
K(p)=-(c_1(p)+c_2(p)),\\
|\nabla u|^2dv_g=|\nabla u|^2dxdy,
\end{align*}
and
\begin{align*}
\frac{\p u}{\p n}dS_g=\frac{\p u}{\p r}rd\theta.
\end{align*}
For $\alpha_i$ and $\beta_i$, we have the following lemma:
\begin{lem}\label{lem-alpha}
It holds that
\begin{align*}
\alpha_i+\beta_i=0,~~i=1,2.
\end{align*}
\end{lem}
\begin{proof}
We have near $x_i$ that
\begin{align*}
2\alpha_i+2\beta_i+O(r)=\Delta_{\mathbb{R}^2}G=0.
\end{align*}
then the lemma is proved by letting $r\to0$.
\end{proof}

We consider the test functions
\begin{align*}
\phi_\ep=
\begin{cases}
w(\frac{x}{\ep})+\lambda_1r\cos\theta+\nu_1r\sin\theta~~&\text{in}~~B_{L\ep}(x_1)\\
G-\eta_1H_{x_1}+4\log L\ep-2\log(1+\pi L^2)-A_{x_1}~~&\text{in}~~B_{2L\ep}(x_1)\setminus B_{L\ep}(x_1)\\
G-\eta_2H_{x_2}+4\log L\ep-2\log(1+\pi L^2)-A_{x_1}~~&\text{in}~~B_{2L\ep}(x_2)\setminus B_{L\ep}(x_2)\\
-w(\frac{x}{\ep})+\lambda_2r\cos\theta+\nu_2r\sin\theta+8\log L\ep\\
~~~~~~~~-4\log(1+\pi L^2)-A_{x_1}+A_{x_2}~~&\text{in}~~B_{L\ep}(x_2)\\
G+4\log L\ep-2\log(1+\pi L^2)-A_{x_1}~~&\text{otherwise.}
\end{cases}
\end{align*}
Here, for $i=1,2$, we have near $x_i$ that
$$H_{x_i}=\alpha_ir^2\cos^2\theta+\beta_ir^2\sin^2\theta+\xi_ir^2\cos\theta\sin\theta +\tau_i(r,\theta)+O(r^4)$$
and
$\eta_i$ is a cut-off function which equals $1$ in $B_{L\ep}(x_i)$, equals $0$ in $M\setminus B_{2L\ep}(x_i)$ and satisfies $|\nabla\eta_i|\leq\frac{C}{L\ep}$.

Now, we calculate $J_{8\pi,8\pi}(\phi_\ep-\overline{\phi_\ep})$. Since $J_{8\pi,8\pi}(u+c)=J_{8\pi,8\pi}(u)$ for any constant $c$, it is equivalent
to calculate $J_{8\pi,8\pi}(\phi_\ep)$. We would like to remark that after calculating $\int_Mh_1e^{\phi_\ep}dv_g$ and $\int_Mh_2e^{-\phi_\ep}dv_g$, we shall see this two terms are both positive, therefore $\phi_\ep-\overline{\phi_\ep}\in\mathcal{E}$.

Firstly, we have
\begin{align}\label{test-energy}
\int_M|\nabla \phi_\ep|^2dv_g=&\int_{B_{L\ep}(x_1)\cup B_{L\ep}(x_2)}|\nabla \phi_\ep|^2dv_g+\int_{M\setminus(B_{L\ep}(x_1)\cup B_{L\ep}(x_2))}|\nabla G|^2dv_g\nonumber\\
&-2\sum_{i=1,2}\int_{B_{2L\ep}(x_i)\setminus B_{L\ep}(x_i)}\nabla G\nabla(\eta_iH_{x_i})dv_g\nonumber\\
&+\sum_{i=1,2}\int_{B_{2L\ep}(x_i)\setminus B_{L\ep}(x_i)}|\nabla(\eta_iH_{x_i})|^2dv_g.
\end{align}
By direct calculations, we have
\begin{align}\label{test-energy-1}
&\int_{B_{L\ep}(x_1)\cup B_{L\ep}(x_2)}|\nabla \phi_\ep|^2dv_g\nonumber\\
=&2\int_{B_L}|\nabla w|^2dxdy+\pi(\lambda_1^2+\nu_1^2+\lambda_2^2+\nu_2^2)(L\ep)^2\nonumber\\
=&32\pi\log(1+\pi L^2)-\frac{32\pi^2L^2}{1+\pi L^2}+\pi(\lambda_1^2+\nu_1^2+\lambda_2^2+\nu_2^2)(L\ep)^2.
\end{align}
Using \eqref{green-2} and integrating by parts, one obtains
\begin{align}\label{test-energy-2}
\int_{M\setminus(B_{L\ep}(x_1)\cup B_{L\ep}(x_2))}|\nabla G|^2dv_g
=-\int_{\p B_{L\ep}(x_1)}G\frac{\p G}{\p n}dS_g-\int_{\p B_{L\ep}(x_2)}G\frac{\p G}{\p n}dS_g.
\end{align}
Since $\int_0^{2\pi}\tau_i(r,\theta)d\theta=\int_0^{2\pi}\frac{\p \tau_i}{\p r}(r,\theta)d\theta=0$, we have
\begin{align*}
\int_{\p B_r(x_i)}G\frac{\p G}{\p n}dS_g
=&\int_0^{2\pi}\left(a_i\log r+A_{x_i}+\lambda_ir\cos\theta+\nu_ir\sin\theta+\alpha_ir^2\cos^2\theta\right.\\
                    &\left.~~~~+\beta_ir^2\sin^2\theta+\xi_ir^2\cos\theta\sin\theta +\tau_i(r,\theta)+O(r^4)\right)\\
              &\times\left(\frac{a_i}{r}+\lambda_i\cos\theta+\nu_i\sin\theta+2\alpha_ir\cos^2\theta\right.\\
                    &\left.~~~~+2\beta_ir\sin^2\theta+2\xi_ir\cos\theta\sin\theta +\frac{\p\tau_i}{\p r}(r,\theta)+O(r^3)\right)rd\theta\\
=&2\pi a_i^2\log r+2\pi a_iA_{x_i}+\pi a_i(\alpha_i+\beta_i)r^2+\pi(\lambda_i^2+\nu_i^2)r^2\\
 &+2\pi a_i(\alpha_i+\beta_i)r^2\log r+2\pi A_{x_i}(\alpha_i+\beta_i)r^2+O(r^4\log r)\\
=&2\pi a_i^2\log r+2\pi a_iA_{x_i}+\pi(\lambda_i^2+\nu_i^2)r^2+O(r^4\log r),
\end{align*}
where in the last equality we have used $\alpha_i+\beta_i=0$ (cf. Lemma \ref{lem-alpha}). Taking this into \eqref{test-energy-2} and noticing $a_1=-a_2=-4$, one obtains
\begin{align}\label{test-energy-3}
\int_{M\setminus(B_{L\ep}(x_1)\cup B_{L\ep}(x_2))}&|\nabla G|^2dv_g=-64\pi\log(L\ep)+8\pi(A_{x_1}-A_{x_2})\nonumber\\
&-\pi(\lambda_1^2+\nu_1^2+\lambda_2^2+\nu_2^2)(L\ep)^2+O((L\ep)^4\log(L\ep)).
\end{align}
Integrating by parts and using \eqref{green-2}, we have
\begin{align}\label{test-energy-4}
&\int_{B_{2L\ep}(x_i)\setminus B_{L\ep}(x_i)}\nabla G\nabla(\eta_iH_{x_i})dv_g\nonumber\\
=&-\int_{\p B_{L\ep}(x_i)}H_{x_i}\frac{\p G}{\p n}dS_g\nonumber\\
=&-\int_0^{2\pi}\left(\alpha_ir^2\cos^2\theta+\beta_ir^2\sin^2\theta+\xi_ir^2\cos\theta\sin\theta +\tau_i(r,\theta)+O(r^4)\right)\nonumber\\
 &~~~~~~~~\times\left(\frac{a_i}{r}+\lambda_i\cos\theta+\nu_i\sin\theta+O(r)\right)r|_{r=L\ep}d\theta\nonumber\\
=&-\pi a_i(\alpha_i+\beta_i)(L\ep)^2+O((L\ep)^4)\nonumber\\
=&O((L\ep)^4),
\end{align}
where in the last equality we have used $\alpha_i+\beta_i=0$. It is clear that
\begin{align}\label{test-energy-5}
\int_{B_{2L\ep}(x_i)\setminus B_{L\ep}(x_i)}|\nabla(\eta_iH_{x_i})|^2dv_g=\int_{B_{2L\ep}(x_i)\setminus B_{L\ep}(x_i)}O(r^2)dv_g=O((L\ep)^4).
\end{align}
Substituting \eqref{test-energy-1}, \eqref{test-energy-3}, \eqref{test-energy-4} and \eqref{test-energy-5} into \eqref{test-energy}, one obtains
\begin{align}\label{test-energy-last}
\int_M|\nabla \phi_\ep|^2dv_g=&32\pi\log(1+\pi L^2)-\frac{32\pi^2L^2}{1+\pi L^2}-64\pi\log(L\ep)\nonumber\\
&+8\pi(A_{x_1}-A_{x_2})+O((L\ep)^4\log(L\ep)).
\end{align}

Denote by $\mathcal{M}_1=\frac{1}{\pi}(-\frac{K(x_1)}{2}+\frac{(\lambda_1+b_1(x_1))^2+(\nu_1+b_2(x_1))^2}{4})$. We have
\begin{align}\label{test-phi-1}
&\int_{B_{L\ep}(x_1)}e^{\phi_\ep}dv_g\nonumber\\
=&\ep^2\int_{B_L}\frac{e^{(\lambda_1+b_1(x_1))\ep x+(\nu_1+b_2(x_1))\ep y+c_1(x_1)\ep^2x^2+c_2(x_1)\ep^2y^2+c_{12}(x_1)\ep^2xy+O((\ep r)^3)}}{(1+\pi r^2)^2}dxdy\nonumber\\
=&\ep^2\left(1-\frac{1}{1+\pi L^2}+\ep^2\mathcal{M}_1\log(1+\pi L^2)+O(\ep^2)\right).
\end{align}
Since $\alpha_i+\beta_i=0$ for $i=1,2$, we have
\begin{align}\label{test-phi-2}
&\int_{B_{\delta}(x_1)\setminus B_{L\ep}(x_1)}e^{\phi_\ep}dv_g\nonumber\\
=&\frac{(L\ep)^4}{(1+\pi L^2)^2}\int_0^{2\pi}\int_{L\ep}^{\delta}\frac{1}{r^4}e^{\lambda_1r\cos\theta+\nu_1r\sin\theta+(1-\eta_1)H_{x_1}}\nonumber\\
&\times e^{b_1(x_1)r\cos\theta+b_2(x_1)r\sin\theta+c_1(x_1)r^2\cos^2\theta+c_2(x_1)r^2\sin^2\theta+c_{12}(x_1)r^2\cos\theta\sin\theta+O(r^3)}rdrd\theta\nonumber\\
=&\frac{(L\ep)^4}{(1+\pi L^2)^2}\int_{L\ep}^{\delta}\left(\frac{2\pi}{r^3}+2\pi^2\mathcal{M}_1\frac{1}{r}+O(1)\right)dr\nonumber\\
=&\ep^2\left(\frac{\pi L^2}{(1+\pi L^2)^2}-\mathcal{M}_1\ep^2\log(L\ep)^2+O(\ep^2)+O\left(\frac{1}{L^4}\right)\right).
\end{align}
It is clear that
\begin{align}\label{test-phi-3}
\int_{M\setminus B_{\delta}(x_1)}e^{\phi_\ep}dv_g=O(\ep^4).
\end{align}
From \eqref{test-phi-1}, \eqref{test-phi-2} and \eqref{test-phi-3} one obtains
\begin{align*}
\int_Me^{\phi_\ep}dv_g=\ep^2\left(1+\mathcal{M}_1\ep^2(-\log\ep^2)+O(\ep^2)+O\left(\frac{1}{L^4}\right)\right).
\end{align*}

For small $\delta>0$, we assume in $B_\delta(x_1)$ that
\begin{align*}
h_1(x)-h_1(x_1)=&k_1r\cos\theta+k_2r\sin\theta\\
&+k_3r^2\cos^2\theta+2k_4r^2\cos\theta\sin\theta+k_5r^2\sin^2\theta+O(r^3).
\end{align*}
By a simple calculation, we obtain
\begin{align*}
&\int_{B_{L\ep}(x_1)}(h_1(x)-h_1(x_1))e^{\phi_\ep}dv_g\\
=&\frac{1}{2\pi}\left[k_3+k_5+k_1(\lambda_1+b_1(x_1))+k_2(\nu_1+b_2(x_1))\right]\ep^4\log(1+\pi L^2)+O(\ep^4)
\end{align*}
and
\begin{align*}
&\int_{M\setminus B_{L\ep}(x_1)}(h_1(x)-h_1(x_1))e^{\phi_\ep}dv_g\\
=&\int_{B_{\delta}(x_1)\setminus B_{L\ep}(x_1)}(h_1(x)-h_1(x_1))e^{\phi_\ep}dv_g+\int_{M\setminus B_{\delta}(x_1)}(h_1(x)-h_1(x_1))e^{\phi_\ep}dv_g\\
=&-\frac{1}{2\pi}\left[k_3+k_5+k_1(\lambda_1+b_1(x_1))+k_2(\nu_1+b_2(x_1))\right]\ep^4\log((L\ep)^2)+O(\ep^4).
\end{align*}
Therefore, we have
\begin{align*}
&\int_{M}h_1e^{\phi_\ep}dv_g\nonumber\\
=&h_1(x_1)\ep^2\left[1+\left(\mathcal{M}_1+\frac{k_3+k_5+k_1(\lambda_1+b_1(x_1))+k_2(\nu_1+b_2(x_1))}{2\pi h_1(x_1)}\right)\ep^2(-\log\ep^2)\right.\\
 &~~~~~~~~~~~~~~~~\left.+O(\ep^2)+O\left(\frac{1}{L^4}\right)\right].
\end{align*}
Notice that
\begin{align*}
\mathcal{M}_1+\frac{k_3+k_5+k_1(\lambda_1+b_1(x_1))+k_2(\nu_1+b_2(x_1))}{2\pi h_1(x_1)}\geq\frac{1}{4\pi}\left[\Delta\log h_1(x_1)-2K(x_1)\right].
\end{align*}
So,
\begin{align}\label{test-hphi}
&\log\int_{M}h_1e^{\phi_\ep}dv_g\geq\log h_1(x_1)+\log\ep^2\nonumber\\
&~~~~~~~~~+\frac{1}{4\pi}\left[\Delta\log h_1(x_1)-2K(x_1)\right]\ep^2(-\log\ep^2)+O(\ep^2)+O\left(\frac{1}{L^4}\right).
\end{align}

Lastly, we calculate $\int_Mh_2e^{-\phi_\ep}dv_g$.

Denote by $\mathcal{M}_2=\frac{1}{\pi}\left(-\frac{K(x_2)}{2}+\frac{(-\lambda_2+b_1(x_2))^2+(-\nu_2+b_2(x_2))^2}{4}\right)$. We have
\begin{align}\label{test--phi-1}
&\int_{B_{L\ep}(x_2)}e^{-\phi_\ep}dv_g\nonumber\\
=&\frac{(1+\pi L^2)^4}{(L\ep)^8}e^{A_{x_1}-A_{x_2}}\ep^2\left(1-\frac{1}{1+\pi L^2}+\ep^2\mathcal{M}_2\log(1+\pi L^2)+O(\ep^2)\right).
\end{align}
By direct calculations and using the fact $\alpha_2+\beta_2=0$, we have
\begin{align}\label{test--phi-2}
&\int_{B_{\delta}(x_2)\setminus B_{L\ep}(x_2)}e^{-\phi_\ep}dv_g\nonumber\\
=&\frac{(1+\pi L^2)^2}{(L\ep)^4}e^{A_{x_1}-A_{x_2}}\left(\frac{\pi}{(L\ep)^2}-\pi^2\mathcal{M}_2\log(L\ep)^2+O(1)\right).
\end{align}
It is clear that
\begin{align}\label{test--phi-3}
\int_{M\setminus B_{\delta}(x_2)}e^{-\phi_\ep}dv_g=O(\ep^{-4}).
\end{align}

Using \eqref{test--phi-1}, \eqref{test--phi-2} and \eqref{test--phi-3} one obtains
\begin{align*}
\int_Me^{-\phi_\ep}dv_g=\pi^4\ep^{-6}e^{A_{x_1}-A_{x_2}}\left(1+\frac{4}{\pi L^2}+\mathcal{M}_2\ep^2(-\log\ep^2)+O(\ep^2)+O\left(\frac{1}{L^4}\right)\right).
\end{align*}

For small $\delta>0$, we assume in $B_\delta(x_2)$ that
\begin{align*}
h_2(x)-h_2(x_2)=&l_1r\cos\theta+l_2r\sin\theta\\
&+l_3r^2\cos^2\theta+2l_4r^2\cos\theta\sin\theta+l_5r^2\sin^2\theta+O(r^3).
\end{align*}

By direct calculations, we obtain
\begin{align*}
&\int_{B_{L\ep}(x_2)}(h_2-h_2(x_2))e^{-\phi_\ep}dv_g=\frac{(1+\pi L^2)^4}{(L\ep)^8}e^{A_{x_1}-A_{x_2}}\\
&\times\frac{1}{2\pi}\left[l_3+l_5+l_1(-\lambda_2+b_1(x_2))+l_2(-\nu_2+b_2(x_2))\right]\ep^4\log(1+\pi L^2)+O(\ep^{-4})
\end{align*}
and
\begin{align*}
&\int_{M\setminus B_{L\ep}(x_2)}(h_2-h_2(x_2))e^{-\phi_\ep}dv_g\\
=&\int_{B_{\delta}(x_2)\setminus B_{L\ep}(x_2)}(h_2-h_2(x_2))e^{-\phi_\ep}dv_g+\int_{M\setminus B_{\delta}(x_2)}(h_2-h_2(x_2))e^{-\phi_\ep}dv_g\\
=&\frac{(1+\pi L^2)^2}{(L\ep)^4}e^{A_{x_1}-A_{x_2}}\nonumber\\
&\times\left(-\frac{1}{2\pi}\left[l_3+l_5+l_1(-\lambda_2+b_1(x_2))+l_2(-\nu_2+b_2(x_2))\right]\right)
\log(L\ep)^2+O(\ep^{-4}).
\end{align*}
Therefore, we have
\begin{align*}
&\int_{M}h_2e^{-\phi_\ep}dv_g=h_2(x_2)e^{A_{x_1}-A_{x_2}}\pi^4\ep^{-6}\nonumber\\
 &\times\left[1+\frac{4}{\pi L^2}+(\mathcal{M}_2+\frac{l_3+l_5+l_1(-\lambda_2+b_1(x_2))+l_2(-\nu_2+b_2(x_2))}{2\pi h_2(x_2)})\ep^2(-\log\ep^2)\right.\\
 &~~~~~~~~~~~~~~~~\left.+O(\ep^2)+O\left(\frac{1}{L^4}\right)\right].
\end{align*}
Notice that
\begin{align*}
\mathcal{M}_2+\frac{l_3+l_5+l_1(-\lambda_2+b_1(x_2))+l_2(-\nu_2+b_2(x_2))}{2\pi h_2(x_2)}\geq\frac{1}{4\pi}\left[\Delta\log h_2(x_2)-2K(x_2)\right].
\end{align*}
So,
\begin{align}\label{test--hphi}
&\log\int_{M}h_2e^{-\phi_\ep}dv_g\geq\log h_2(x_2)+(A_{x_1}-A_{x_2})+4\log\pi-6\log\ep+\frac{4}{\pi L^2}\nonumber\\
&~~~~~~~~~+\frac{1}{4\pi}\left[\Delta\log h_2(x_2)-2K(x_2)\right]\ep^2(-\log\ep^2)+O(\ep^2)+O\left(\frac{1}{L^4}\right).
\end{align}

Finally, by \eqref{test-energy-last}, \eqref{test-hphi} and \eqref{test--hphi}, we have
\begin{align}\label{uper-bound2}
J_{8\pi,8\pi}(\phi_\ep)\leq&16\pi\log(1+\pi L^2)-\frac{16\pi^2L^2}{1+\pi L^2}-32\pi\log(L\ep)\nonumber\\
&+4\pi(A_{x_1}-A_{x_2})+O((L\ep)^4\log(L\ep))\nonumber\\
&-8\pi\log h_1(x_1)-8\pi\log\ep^2\nonumber\\
&-2\left[\Delta\log h_1(x_1)-2K(x_1)\right]\ep^2(-\log\ep^2)+O(\ep^2)+O\left(\frac{1}{L^4}\right)\nonumber\\
&-8\pi\log h_2(x_2)-8\pi(A_{x_1}-A_{x_2})-32\pi\log\pi+48\pi\log\ep-\frac{32}{L^2}\nonumber\\
&-2\left[\Delta\log h_2(x_2)-2K(x_2)\right]\ep^2(-\log\ep^2)+O(\ep^2)+O\left(\frac{1}{L^4}\right)\nonumber\\
=&-16\pi\log\pi-16\pi-4\pi(A_{x_1}-A_{x_2})-8\pi\log h_1(x_1)-8\pi\log h_2(x_2)\nonumber\\
&-2\left[\Delta\log h_1(x_1)-2K(x_1)+\Delta\log h_2(x_2)-2K(x_2)\right]\ep^2(-\log\ep^2)\nonumber\\
&+O((L\ep)^4\log(L\ep))+O(\ep^2)+O\left(\frac{1}{L^4}\right).
\end{align}
Letting $L^4\ep^2=\frac{1}{\log(-\log\ep)}$ in \eqref{uper-bound2}, we obtain
\begin{align*}
J_{8\pi,8\pi}(\phi_\ep)\leq&-16\pi\log\pi-16\pi-4\pi(A_{x_1}-A_{x_2})-8\pi\log h_1(x_1)-8\pi\log h_2(x_2)\nonumber\\
&-2\left[\Delta\log h_1(x_1)-2K(x_1)+\Delta\log h_2(x_2)-2K(x_2)\right]\ep^2(-\log\ep^2)\nonumber\\
&+o(\ep^2(-\log\ep^2)).
\end{align*}
Under assumption \eqref{cond-2}, one has
$$\Delta\log h_1(x_1)-2K(x_1)+\Delta\log h_2(x_2)-2K(x_2)>0.$$
Then for sufficiently small $\ep>0$, we have
\begin{align*}
J_{8\pi,8\pi}(\phi_\ep)<-16\pi\log\pi-16\pi-4\pi(A_{x_1}-A_{x_2})-8\pi\log h_1(x_1)-8\pi\log h_2(x_2).
\end{align*}
This together with \eqref{lower2-bound} proves Theorem \ref{thm-2}. $\hfill{\square}$

\end{document}